\documentclass[english]{article} 
\usepackage[utf8]{inputenc}
\usepackage[T1]{fontenc}
\usepackage{lmodern}
\usepackage[a4paper]{geometry}
\usepackage{enumitem}
\usepackage{verbatim}
\usepackage{graphicx}
\usepackage{amsmath}
\usepackage{amsthm}
\usepackage{amsfonts}
\usepackage{amssymb}
\usepackage{amsmath}
\usepackage{mathrsfs}
\usepackage{framed}
\usepackage{tikz}
\usepackage{pgfplots}
\pgfplotsset{compat=1.15}
\usetikzlibrary{intersections}
\usepackage[english]{babel}
\usepackage[citecolor=blue,colorlinks=true]{hyperref}
\usepackage{titlesec}
\usepackage{esint}

\theoremstyle{plain}
\newtheorem{theorem}{Theorem}[section]
\newtheorem{proposition}[theorem]{Proposition}
\newtheorem{lemma}[theorem]{Lemma}
\newtheorem{corollary}[theorem]{Corollary}
\newtheorem{proposition-definition}[theorem]{Proposition-Definition}

\theoremstyle{definition}
\newtheorem{definition}{Definition}[section]

\theoremstyle{remark}
\newtheorem{remark}{Remark}[section]

\usepackage{bbm}

\numberwithin{equation}{section}
\newcommand{\Z}{\mathbb{Z}}

\newcommand{\R}{\mathbb{R}}

\def\ocirc#1{\ifmmode\setbox0=\hbox{$#1$}\dimen0=\ht0
    \advance\dimen0 by1pt\rlap{\hbox to\wd0{\hss\raise\dimen0
    \hbox{\hskip.2em$\scriptscriptstyle\circ$}\hss}}#1\else
    {\accent"17 #1}\fi} 

\newcommand{\eps}{\varepsilon}
\newcommand{\PP}{\mathbb{P}}
\newcommand{\E}{\mathbb{E}}

\newcommand{\dual}[2]{\langle #1, #2\rangle}
\newcommand{\duall}[2]{\langle #1, #2\rangle}

\newcommand{\ttt}{\mathbf{t}}
\newcommand{\sss}{\mathbf{s}}
\newcommand{\uuu}{\mathbf{u}}
\newcommand{\vvv}{\mathbf{v}}

\newcommand{\nd}{I}
\newcommand{\snd}{J}
\newcommand{\ray}{\rho}
\newcommand{\cv}{a} 

\newcommand{\sig}{\mathbf{\sigma}}

\newcommand{\di}{\mathrm{d}}
\newcommand{\dt}{\mathrm{d}t}

\newcommand{\ds}{\mathrm{d}s}

\newcommand{\divv}{\mathrm{div}}

\DeclareMathOperator{\dist}{dist}

\usepackage{hyperref}

\usepackage[textsize=tiny]{todonotes}

\usepackage{authblk}

\begin{document}

\title{A quantitative averaging lemma for spatially dependent vector fields}
\author{Paul Alphonse, Billel Guelmame and Julien Vovelle}

\date{\today}
\maketitle
\begin{abstract}
	We prove a quantitative averaging lemma for spatially dependent vector fields. Our proof is based on an iteration of the regularizing operator and some elementary considerations about the local inversion theorem.
\end{abstract}

\medskip

\medskip


\tableofcontents

\section{Introduction}

\subsection{Context}

The purpose of this paper is to give a \emph{quantitative} averaging lemma for \emph{spatially dependent} vector fields, using elementary tools of real analysis, see Theorem~\ref{th:AVLemmaII} below. There is a  very long list of statements which can be termed ``averaging lemma''. Such kind of results establish a gain of regularity, or at least some compactness properties, by gathering (typically by means of an integration procedure) a punctual information on directional derivatives. Depending on the kind of equations and the functional framework that are considered, many different results can be obtained. This is illustrated in the review paper \cite{Arsenio2015} and \cite{ArsenioLerner2021} for the sole case of averaging lemmas for a kinetic transport operator $v\cdot\nabla_x$. However, our interest is specifically in averaging lemmas for equations involving space-dependent vector fields $a(x,v)\cdot\nabla_x$, so, besides \cite{Arsenio2015,ArsenioLerner2021} and references therein, we will solely make reference to the following works:
\begin{itemize}
	\item  the very first papers \cite{Agoshkov1984,GolsePerthameSentis1985,GolseLionsPerthameSentis1988,DiPernaLionsMeyer1991} on the subject, 
	\item the paper \cite{Gerard91}, where the notion of microlocal defect measure is introduced and in particular applied to the derivation of averaging lemmas with non constant (\textit{i.e.} space-dependent) coefficients, see  \cite[Theorem~2.5]{Gerard91} (see also the extension given in \cite{GerardGolse92}),
	\item various results by Mitrovi\'c and co-authors: \cite{LazarMitrovic2012} in an $L^2$-framework, \cite{LazarMitrovic2017} and \cite{ErcegMisurMitrovic2023} in a $L^p$-framework,
	\item the papers \cite{GerardGolse92} by G\'erard and Golse, and the recent preprint \cite{ErcegKarlsenMitrovic25} by Erceg, Karlsen and Mitrovi\'c.
\end{itemize}

The averaging lemmas in \cite{ErcegMisurMitrovic2023,Gerard91,LazarMitrovic2012,LazarMitrovic2017} are of qualitative nature (compactness properties on the average of the unknown). In the last references \cite{GerardGolse92} and \cite{ErcegKarlsenMitrovic25} (as in our main result Theorem~\ref{th:AVLemmaII}) some quantitative results are obtained. To put it in a very concise way, a gain of $1/2$ derivatives  in a $L^2$-framework is obtained in \cite{GerardGolse92}, and a positive gain of derivative in a Lebesgue space is obtained in \cite{ErcegKarlsenMitrovic25}. Both results are based on relatively elaborated tools from Fourier analysis (combined with a doubling of variable argument in \cite{ErcegKarlsenMitrovic25}).

\subsection{Main results}

We consider typically the solution $f(x,v)$ to the equation
\[
	f(x,v)-a(x,v)\cdot\nabla_x f(x,v)=g(x,v),\quad (x,v)\in\R^d\times\R^m,
\]
where $g$ is some given function in $L^p(\R^d\times\R^m)$, $p\in(1,+\infty)$, and, for each $v\in\R^m$, $x\mapsto a(x,v)$ is a vector field on $\R^d$. We will prove that, if the vector-field $a(\cdot,v)$ is sufficiently varying with $v$, then the ``average''
\begin{equation}\label{average rho psi}
    \varrho_\psi(x):=\int_{\R^m}f(x,v)\psi(v)\,\mathrm dv
\end{equation}
has some regularity properties that can be quantified in a fractional Sobolev space $W^{\beta,p}(\R^d)$. In \eqref{average rho psi}, the test-function is typically in $C^\infty_c(\R^m)$ (although less regularity and localisation properties are sufficient to get a result). To emphasize the averaging procedure (and also to emphasize the point of view that we adopt here, which is to iterate $d$ times a run-and-tumble dynamics, see Remark~\ref{rem:run-and-tumble}), we will consider the following framework, where the velocity space is viewed as a probability space.

Let $(\Omega,\PP)$ be a probability space, let $d\geq 1$ and let  $\cv\colon\Omega\to C^2_b(\R^d;\R^d)$ be a random vector field such that, almost surely,
\begin{equation}\label{definition norme}
    \|\cv\|_{C^2_b(\R^d;\R^d)} := \sum_{|\gamma|\leq 2}\Vert\partial^\gamma_x\cv\Vert_{L^{\infty}(\mathbb R^d)} \leq M,
\end{equation}
for a constant $M> 0$. We assume that $\cv$ is non-stationary in the following sense: there exists a constant $C>0$ and $\alpha\in(0,1]$ such that, for all $x\in\R^d$, all $\varepsilon\in(0,1)$ and all directions $\sigma\in\R^d$ such that $|\sigma|=1$,
\begin{equation}\label{nd-a}
	\PP(|\cv(x)\cdot\sigma|<\eps)\leq C\eps^\alpha.
\end{equation}
Let $p\in(1,+\infty)$ and $f,g\in L^p(\Omega\times\mathbb R^d)$ satisfying the following equation: almost surely,
\begin{equation}\label{f plus agrad f}
	f-\cv(x)\cdot\nabla_xf = g\quad \text{in $\mathcal{D}'(\R^d)$.}
\end{equation}
Let us finally define
\begin{equation}\label{value delta}
	\delta(p,\alpha,d):=\frac{2}{\max(p,p')} \frac1{4d}\frac{1}{1+\frac{2d}{\alpha}} > 0,
\end{equation}
where  $p'\in(1,+\infty)$ is given by the relation $1/p+1/p' = 1$. 

The main result of the paper is that when $p=2$, the average of $f$ has almost $\delta(2,\alpha,d)$-regularity, and it is stated as follows.

\begin{theorem}[Averaging lemma]\label{th:AVLemmaII} 
 Assume that $p=2$. Then, there exists a positive constant $r_0=r_0(d,\alpha)>0$, such that, for any $\theta < \delta(2,\alpha,d)$, the average $\varrho:=\E\left[f\right]$ satisfies the estimate 
\begin{equation}\label{AV-lemII}
	\forall r\in(0,r_0], \quad\sup_{|y| \le r}\|\tau_y\varrho-\varrho\|_{L^2(\R^d)}\leq \kappa (\|f\|_{L^2(\Omega\times\R^d)} + \|g\|_{L^2(\Omega\times\R^d)}) r^{\theta},
\end{equation}
with $\tau_y\varrho = \varrho(\cdot-y)$, and where the constant $\kappa$ depends on $\Vert\divv_x(\cv)\Vert_{L^{\infty}}$, $M$, $C$, $\alpha$, $\theta$ and $d$, and is independent on $f$ and $g$.
\end{theorem}

The restriction to $p=2$ in Theorem \eqref{th:AVLemmaII} is due to the lack of an inversion formula for the equation \eqref{f plus agrad f}. However, assuming the following extra assumption on the divergence of the vector field $\cv$
\begin{equation}\label{divergence smallness}
	\|\divv_x(\cv)\|_{L^{\infty}(\R^d)} \leq \lambda <   c(p) := 
	\begin{cases}
   p, & \text{when $p \geq 2$}, \\
	1, & \text{when $p<2$}.
	\end{cases}
\end{equation}
we can improve the estimate \eqref{AV-lemII} and also establish $L^p$-averaging lemma for the function $f$.

\begin{corollary}\label{cor:avlem}
Assume that the smallness condition \eqref{divergence smallness} holds. Then, there exists a positive constant $r_0=r_0(d,\alpha)>0$, such that, for any $\theta < \delta(p,\alpha,d)$, the average $\varrho:=\E\left[f\right]$ satisfies the estimate 
\begin{equation}\label{AV-lemIII}
	\forall r\in(0,r_0], \quad\sup_{|y| \le r}\|\tau_y\varrho-\varrho\|_{L^p(\R^d)}\leq \kappa \|g\|_{L^p(\Omega\times\R^d)} r^{\theta},
\end{equation}
with $\tau_y\varrho = \varrho(\cdot-y)$, and where the constant $\kappa$ depends on $\Vert\divv_x(\cv)\Vert_{L^{\infty}}$, $M$, $C$, $\alpha$, $p$, $\theta$ and $d$, and is independent on $f$ and $g$.
\end{corollary}

Some remarks regarding Theorem \ref{th:AVLemmaII} are now in order. 

\medskip

\begin{remark}
While proving Theorem \ref{th:AVLemmaII}, we will use the following equivalent form of \eqref{nd-a}: for all vector space $V\subset\R^d$ of dimension $<d$, for all $x\in\R^d$ and for all $\varepsilon\in(0,1)$,
\begin{equation}\label{nd-a-dist}
	\PP(\dist(\cv(x),V)<\eps)\leq C\eps^\alpha,
\end{equation}
where 
\[
	\dist(\cv,V)=\inf_{b\in V}|\cv-b|.
\]
This non-degeneracy hypothesis \eqref{nd-a} on the vector field $\cv$ is a quantitative version of that of \cite[condition (2.12)]{Gerard91}, or \cite[condition (11)]{LazarMitrovic2012} modulo the fact that Condition~(11) in \cite{LazarMitrovic2012} can be only a.e. in $x$. It also corresponds to Condition (2.1) in \cite{GolseLionsPerthameSentis1988} in the case where $\cv$ does not depend on the variable $x\in\mathbb R^d$.
\end{remark}

\medskip

\begin{remark}
Under the assumption \eqref{divergence smallness}, an immediate consequence of \eqref{AV-lemIII} is the following regularity estimates for the function $\varrho$ in the Sobolev spaces $ W^{\theta,p}(\mathbb R^d)$ for every $0<\theta<\delta(p,\alpha,d)$,
\[
    \Vert\varrho\Vert_{W^{\theta,p}(\R^d)}\leq \kappa \|g\|_{L^p(\Omega\times\R^d)},
\]
where the constant $\kappa$ may be different from the constant $\kappa$ in \eqref{AV-lemIII}, but has the same dependence with respect to the various parameters. When $p=2$, the exponent $\delta(2,\alpha,d)$ is far below from the expected $\alpha/2$ order, which is the gain of derivatives obtained in the case where $\cv$ does not depend on the space variable $x\in\mathbb R^d$ (see Proposition \ref{a inde} in Appendix \ref{sec:aconstant}). In the case $\alpha=1$ and $p=2$, the optimal order $1/2$ is obtained in \cite{GerardGolse92}. An estimate on $\varrho$ in the Sobolev space $W^{\beta,q}_{\mathrm{loc}}(\mathbb R^d)$ is established in \cite{ErcegKarlsenMitrovic25}. The exponents $\beta$ and $q$ satisfy $\beta\in(0,1)$, $q\in (1,p)$, but are not more explicit in \cite{ErcegKarlsenMitrovic25}, contrary to the result we obtain (the value $\delta(p,\alpha,d)$ being given by \eqref{value delta}). However, a more general right-hand side can be considered in \cite{ErcegKarlsenMitrovic25} (with potential application, via the kinetic formulation, to regularity results for solutions to non-linear scalar conservation laws, as in \cite{PerthameSouganidis98}). Also, we use a $C^2$-spatial regularity on the vector-field, while \cite{ErcegKarlsenMitrovic25} (see Remark~3.3 in \cite{ErcegKarlsenMitrovic25}) is less demanding.
\end{remark}

\medskip

\begin{remark}
There is an equivalence between the equation in divergence form
\begin{equation}\label{div form}
    f-\divv_x (\cv f)=g,
\end{equation}
and the transport equation with modified right-hand side
\begin{equation}\label{div form as transport form}
    f-\cv(x)\cdot\nabla_x f=g+ \divv(\cv)f,
\end{equation}
both equations \eqref{div form} and \eqref{div form as transport form} being understood almost surely, in $\mathcal{D}'(\R^d)$. Therefore, our strategy of proof of Theorem \ref{th:AVLemmaII} would also lead to the estimate \eqref{AV-lemII} for the solution $f$ of the equation \eqref{div form}.
\end{remark}

\paragraph{Notations} The following notations and conventions will be used throughout this work:
\begin{enumerate}[label={$\cdot$},leftmargin=* ,parsep=2pt,itemsep=0pt,topsep=2pt]
	\item Given a vector $\ttt = (t_1,\ldots,t_d)\in\mathbb R^d$ and $i\in\{1,\ldots,d\}$, we set $\ttt_i = (t_1,\ldots,t_i)\in \mathbb R^i$.
	\item On the space $\mathbb R^d$, the canonical Euclidean norm is denoted $\vert\cdot\vert$, while the notation $\vert\cdot\vert_{\ell^1}$ stands for the  $\ell^1$-norm defined by
	\[
		\vert x\vert_{\ell^1} = \vert x_1\vert+\cdots+\vert x_d\vert,\quad x = (x_1,\ldots,x_d)\in\mathbb R^d.
	\]
	\item Given $R>0$, the open ball centered at $0$ with radius $R$ in $(\mathbb R^d,\vert\cdot\vert)$ is denoted $B(0,R)$.
	\item For every integers $m,n\geq1$ and every matrix $A\in M_{m,n}(\mathbb R)$, the transpose of $A$ is denoted $A^T$. In the case $m=n$, the determinant of $A$ is denoted $\det(A)$ or $\det A$ to alleviate the notation.
	\item The subordinate norm of every multilinear form on a normed vector space $(E,\Vert\cdot\Vert)$ is also denoted $\Vert\cdot\Vert$.
	\item Given a measure space $(X,\mathcal F,\mu)$, the standard inner product on the space $L^2(X)$ is denoted $\langle\cdot,\cdot\rangle_{L^2}$, while $\Vert\cdot\Vert_{L^2}$ stands for the associated norm. The other standard $L^p$-norms are denoted $\Vert\cdot\Vert_{L^p}$. Sometimes, we also additionally indicate in index the current variable of $X$, for instance $L^p_x$ stands for $L^p(\R^d)$ while $L^p_{\omega,x}$ stands for $L^p(\Omega\times\R^d)$.
	\item In most of the proofs, the letter $C$ stands for a positive constant that can change for a line to another. The dependance with respect to some parameter, as the dimension $d$ or the vector field $\cv$ for example, is denoted $C(d)$, $C(\cv)$ or $C(d,\cv)$ depending on the context.
\end{enumerate}

\section{Proof of the averaging lemma}

The aim of this section is to prove Theorem \ref{th:AVLemmaII}.

\subsection{Vector field with small divergence}


First of all, we assume that Theorem \ref{th:AVLemmaII} has been proven, and we explain how Corollary~\ref{cor:avlem} can be then deduced. When the condition \eqref{divergence smallness} holds, it follows from Proposition \ref{prop:weaksol} that the function $f$ satisfies the following bound
\begin{equation}\label{eq:contract}
	\Vert f\Vert_{L^p_x} \le \frac{p}{p-\lambda}\Vert g\Vert_{L^p_x}. 
\end{equation}
In the case where $p=2$, this has for consequence to improve the estimate \eqref{AV-lemII} as follows
\begin{equation}\label{eq:improveavlem}
    \forall r\in(0,r_0], \quad\sup_{|y| \le r}\|\tau_y\varrho-\varrho\|_{L^2_x}\leq \kappa \|g\|_{L^2_{\omega,x}} r^{\theta}.
\end{equation}
Moreover, by Jensen's inequality, we also deduce from \eqref{eq:contract} that
\begin{equation}\label{averaged contraction Lp}
	\Vert\varrho\Vert_{L^p_x} \leq \frac{p}{p-\lambda}\Vert g\Vert_{L^p_{\omega,x}},
\end{equation}
which implies the following inequality for every $y\in\R^d$ by the triangle inequality
\begin{equation}\label{averaged contraction Lp omega}
	\|\tau_y\varrho-\varrho\|_{L^p_x} \leq \frac{2p}{p-\lambda}\Vert g\Vert_{L^p_{\omega,x}}.
\end{equation}
We apply now the Riesz-Thorin theorem to the linear operators $T_y:g\mapsto \tau_y\varrho-\varrho$. Notice that these functions are well-defined since we know by Proposition \ref{prop:weaksol} that when the smallness condition \eqref{divergence smallness} holds, the stationary transport equation \eqref{f plus agrad f} has a unique solution. If $p > 2$, we interpolate \eqref{eq:improveavlem} with \eqref{averaged contraction Lp omega} for $p=\infty$; if $p < 2$, we interpolate \eqref{eq:improveavlem} with \eqref{averaged contraction Lp omega} for $p=1$, and obtain \eqref{AV-lemIII} for every $\theta<\delta(p,\alpha,d)$, with 
\[
	\delta(p,\alpha,d)=\frac{2}{\max(p,p')}\delta(2,\alpha,d).
\]

\subsection{Some reductions}\label{subsec:some reductions}

\paragraph{Resolvent formula.} Let $f,g\in L^2(\Omega\times\mathbb R^d)$. Assume that $f$ is a weak solution of the equation \eqref{f plus agrad f}. The study of such weak solutions for the equation \eqref{f plus agrad f} is performed in Appendix \ref{sec:stationary transport}. In particular, we know how to invert the equation \eqref{f plus agrad f} and get a formula relating the functions $f$ and $g$ only under the extra assumption \eqref{divergence smallness} on the divergence of the vector field $\cv$. In order to reduce our problem to this situation, we introduce a positive constant $\Lambda>0$ and rewrite the equation \eqref{f plus agrad f} as follows
\[
    f-\frac{1}{\Lambda}\cv(x)\cdot\nabla_xf = \frac{1}{\Lambda}g +\bigg(1-\frac{1}{\Lambda}\bigg)f.
\]
When $\Lambda>\lambda/c(p)$, where the constant $c(p)>0$ is defined in \eqref{divergence smallness}, it follows from Proposition \ref{prop:weaksol} that the equation
\begin{equation}\label{eq:modified eq}
    f_{\Lambda}-\cv_{\Lambda}(x)\cdot\nabla_xf_{\Lambda} = h_{\Lambda},
\end{equation}
where we set
\begin{equation}\label{eq:functionh}
    \cv_{\Lambda} = \frac1{\Lambda}a\quad\text{and}\quad h_{\Lambda} = \frac{1}{\Lambda}g +\bigg(1-\frac{1}{\Lambda}\bigg)f,
\end{equation}
has a unique weak solution $f_{\Lambda}\in L^2(\mathbb R^d)$, given almost surely by the resolvent formula \eqref{eq:resolventestimate}. Since the function $f$ is solution of the equation \eqref{eq:modified eq} by assumption, it follows that $f = f_{\Lambda}$ and that $f$ is explicitly given by 
\begin{equation}\label{explicit formula f plus agrad f}
	f = \int_0^\infty e^{-t}h_{\Lambda}\circ\Phi^{\Lambda}_t\,\mathrm dt,
\end{equation}
where $(\Phi_t^{\Lambda})_{t\in\mathbb R}$ is the flow generated by the vector field $\cv_{\Lambda}$.


\paragraph{Reduction to the case $M < 1/(2+d)$.} The transformation \eqref{eq:modified eq}, will be used not only to justify the reduction to a vector field with divergence satisfying \eqref{divergence smallness} and to get the formula \eqref{explicit formula f plus agrad f}. Indeed, while proving Theorem \ref{th:AVLemmaII}, we will need to assume that $\Lambda$ is sufficiently large compared to $M$ (where the quantity $M$ introduced in \eqref{definition norme} gives the almost sure bound on the $C^2$-norm of the vector field $\cv$), in particular to ensure that the condition
\begin{equation}\label{reduction M lambda}
    \frac{\lambda}{\Lambda}\leq\frac{M}{\Lambda} < \frac1{2+d} \ \bigg(\!<\frac13\bigg)
\end{equation}
is satisfied. The control on the size of $M/\Lambda$ will also be invoked at the very end of the proof, see \eqref{theta VS bar theta}.
In the following, though, in order to avoid too heavy formulas, we will not specify the dependence of the function $h_{\Lambda}$ and of the vector field $\cv_{\Lambda}$ with respect to the parameter $\Lambda$. Moreover, at the price of a huge abuse of notation, we will keep using the notations $\lambda$ and $M$ to denote the rescaled quantities $\lambda/\Lambda$ and $M/\Lambda$, respectively.

Note well that this trick, which consists in replacing $a$ by $a_{\Lambda}$, modifies the value of the constant $C$ in the non-degeneracy condition \eqref{nd-a}, $C$ being increased to $C\Lambda^\alpha$.

Let us finally mention that, in the following, we will widely use the fact that $M\le 1$ without stating it explicitly each time this is the case.

\subsection{Duality and \texorpdfstring{$T T^*$}{} argument}

\paragraph{Duality.} Recall that $\tau_y\varrho(x)=\varrho(x-y)$. Our aim is to estimate 
\begin{equation*}
	\langle \tau_{-y}\varrho-\varrho,w\rangle_{L^2_x},
\end{equation*}
where $w$ is an arbitrary element of $L^2(\R^d)$. We expand, using \eqref{explicit formula f plus agrad f}
\begin{equation} \label{duality modulus expand}
	\dual{\varrho-\tau_{-y}\varrho}{w}_{L^2_x}=\dual{\varrho}{w-\tau_y w}_{L^2_x}=\E\left[\dual{f}{w-\tau_y w}_{L^2_x}\right]=\E\left[\dual{h}{ T \tilde{w}}_{L^2_x}\right],
\end{equation}
with
\[
	\tilde{w} :=w-\tau_y w, \quad T w  :=\int_0^\infty e^{-t}(w\circ\Phi_{-t})\, J\Phi_{-t}\, \dt,
\]
where we used a change of variables with the formula \eqref{determinantII}. From \eqref{duality modulus expand} and the Cauchy-Schwarz inequality, it follows that
\[
	\vert\dual{\varrho-\tau_{-y}\varrho}{w}_{L^2_x}\vert \leq \|h\|_{L^2_{\omega,x}} \| T\tilde{w}\|_{L^2_{\omega,x}},
\]
and since 
\[
    \|T  \tilde{w}\|_{L^2_{\omega,x}}^2=\E\left[ \duall{T  \tilde{w}}{T  \tilde{w}}_{L^2_x} \right]=\dual{\Theta\tilde{w}}{\tilde{w}}_{L^2_x},
\]
where
\begin{equation}\label{operator Theta}
	\Theta w := T^*T w,
\end{equation}
we obtain
\begin{equation}\label{eq:intereq}
	\vert\dual{\varrho-\tau_{-y}\varrho}{w}_{L^2_x}\vert
	\leq 2^{1/2}\|h\|_{L^2_{\omega,x}} \|w\|_{L^2_x}^{1/2} \|\Theta (w-\tau_y w)\|_{L^2_x}^{1/2}.
\end{equation}
Our goal is now to prove the estimate
\[
	\sup_{|y|\le r}\|\Theta (w-\tau_y w)\|_{L^2_x}\leq \frac12 \kappa^2 r^{2 \theta}  \|w\|_{L^2_x},
\]
where we set $ \theta =  \theta(2,\alpha,d, \nu)$ to simplify the writing. Using the formula \eqref{determinantII}, one can check that the adjoint of the operator $T :L^2(\mathbb R^d)\rightarrow L^2(\Omega\times\mathbb R^d)$ is given by
\[
    T ^* w^* = \int_0^{+\infty}e^{-t}\E[ w^*\circ\Phi_t]\, \dt,\quad w^*\in L^2(\Omega\times\mathbb R^d).
\]
Recalling the definition \eqref{operator Theta} of $\Theta$, we therefore deduce that
\begin{equation}\label{etape vers theta}
    \Theta w = \iint_{(0,+\infty)^2}e^{-(t+s)}\E[(w\circ\Phi_{t-s})\, (J\Phi_{-s}) \circ \Phi_t]\, \ds\, \dt.
\end{equation}
Performing the change of variables $t' = t-s$ and $s' = s$ then allows to obtain from \eqref{etape vers theta} that
\begin{equation}\label{def Theta}
    \Theta w = \int_{\R}e^{-|t|} \E[(w\circ\Phi_{t})E(t,\cdot)]\, \dt,
\end{equation}
where, denoting by $t^- := \max\{0,-t\}$ the negative part of $t$,
\begin{equation}\label{E def}
   E(t,\cdot) :=  \int_{t^-}^{+\infty} e^{-2(s-t^-)}\, (J\Phi_{-s}) \circ \Phi_{t+s}\, \ds.
\end{equation}
In view of \eqref{divergence smallness} and \eqref{determinantII}, for each $t\in \R$ and $x \in \R^d$, we have almost surely
\begin{equation}\label{E bound}
    \frac{e^{-\lambda t^- }}{2+\lambda} \leq 	E(t,x) \leq \frac{e^{\lambda t^-}}{2-\lambda}.
\end{equation}

\paragraph{Iteration.} The operator $\Theta\colon L^2(\R^d)\to L^2(\R^d)$ is selfadjoint. Introduce the quantity
\[
	\vartheta_k=\|\Theta^{k} \tilde w\|_{L^2_x},\quad \tilde{w}:=w-\tau_y w.
\]
We have the following interpolation bounds: for all $j\geq1$, and $k\in\{0,\ldots,j\}$,
\begin{equation}\label{estim theta level j}
	\vartheta_k\leq \vartheta_j^{k/j}\vartheta_0^{1-k/j}.
\end{equation}
This can be proved by recursion on $j$: for $j\geq1$, we can expand
\[
	\vartheta_j=\dual{\Theta^{j+1}\tilde{w}}{\Theta^{j-1}\tilde{w}}^{1/2}_{L^2_x}
	\leq\vartheta_{j+1}^{1/2}\vartheta_{j-1}^{1/2},
\]
and using \eqref{estim theta level j} for $k=j-1$ obtain 
\[
	\vartheta_j\leq \vartheta_{j+1}^{1/2}\vartheta_j^{1/2-1/2j}\vartheta_0^{1/2j}
	\Longrightarrow\vartheta_j\leq \vartheta_{j+1}^{j/(j+1)}\vartheta_0^{1/(j+1)}.
\]
We report the last estimate in \eqref{estim theta level j} to conclude to \eqref{estim theta level j} at level $j+1$. It follows that 
\begin{equation}\label{theta by higher powers}
	\|\Theta(w-\tau_y w)\|_{L^2_x}\leq (2\|w\|_{L^2_x})^{1-1/j}\|\Theta^j(w-\tau_y w)\|_{L^2_x}^{1/j}.
\end{equation}
We apply \eqref{theta by higher powers} with $j=d$ to reduce our problem to the proof of the following estimate:
\begin{equation}\label{key estim Theta 2}
	\sup_{|y|\le r}\|\Theta^d (w-\tau_y w)\|_{L^2_x}\leq  2^{1-2d} \kappa^{2d} r^{2d\theta}  \|w\|_{L^2_x}.
\end{equation}

\begin{remark}[Run-and-tumble]\label{rem:run-and-tumble}
Notice from \eqref{def Theta} that, if $\cv$ is divergence-free, \textit{i.e.} if $\divv(\cv) = 0$ almost surely, then  $E\equiv1/2$ and so $\Theta$ is simply given by
\[
	\Theta w =\int_\R\frac{e^{-|t|}}{2}\E[w\circ\Phi_t]\, \dt.
\]
Let $(X_n)_{n\geq1}$ be the discrete Markov process defined as follows (note that this is a particular instance of motion by run-and-tumble): knowing $X_n$, 
\begin{enumerate}
	\item draw a sign $\pm$ with equi-probability $1/2$,
	\item draw independently an exponential time variable $\tau\sim\mathcal{E}(1)$ and a vector field $\cv$,
	\item follow the integral curve of $\pm\cv$ until reaching the point $X_{n+1}$ at time $\tau$.
\end{enumerate}	
Then $\Theta$ is the transition operator associated to $X_n\to X_{n+1}$ and the iterate $\Theta^d w$ is thus
\[
	\Theta^d w(x)=\E_x\left[w(X_d)\right],
\]
where the average is done with respect to the successive choices of $\pm$, $\tau$, $\cv$ at steps $1,2,\ldots,d$.
\end{remark}

\paragraph{Estimates after $d$ iterations.} By iteration of \eqref{def Theta}, the expression of $\Theta^d$ is 
\begin{equation}\label{Theta circ d}
    \Theta^d w(x) = \int_{\R^d}e^{-|\ttt|_{\ell^1}} \E\bigg[w\big(\Phi^{(d)}_{t_d}\circ\dotsb\circ\Phi^{(1)}_{t_1}(x)\big) \prod_{i=1}^d E\big(t_i, \Phi^{(i-1)}_{t_{i-1}}\circ\dotsb\circ\Phi^{(1)}_{t_1} (x) \big)\bigg]\,\mathrm d\ttt,
\end{equation}
where  $\ttt=(t_1,\ldots,t_d)$, and $(\Phi^{(k)}_t)_{t\in\mathbb R}$ is the flow associated to a copy $\cv^{(k)}$ of $\cv$, the family $\{\cv^{(k)}: 1\leq k\leq d\}$ being independent. In \eqref{Theta circ d}, we use the convention that when $i=1$,
\[
	\Phi^{(i-1)}_{t_{i-1}}\circ\dotsb\circ\Phi^{(1)}_{t_1} = \mathrm{Id}.
\]
In order to simplify the writing, let us define
\begin{equation}\label{K definition}
	K(\ttt,x) := \prod_{i=1}^d E\big(t_i,\Phi^{(i-1)}_{t_{i-1}}\circ\dotsb\circ\Phi^{(1)}_{t_1}(x)\big).
\end{equation}
With this notations, using \eqref{Theta circ d}, we write
\begin{multline}\label{Theta circ d delta w}
	\Theta^d(w-\tau_y w)(x) = 
	\int_{\R^d} e^{-|\ttt|_{\ell^1}} \E\big[w\big(\Phi^{(d)}_{t_d}\circ\dotsb\circ\Phi^{(1)}_{t_1}(x)\big) K(\ttt,x)\big]\,\mathrm d\ttt \\
    -\int_{\R^d} e^{-|\ttt|_{\ell^1}} \E\big[w\big(\Phi^{(d)}_{t_d}\circ\dotsb\circ\Phi^{(1)}_{t_1}(x)-y\big) K(\ttt,x)\big]\,\mathrm d\ttt.
\end{multline} 
At this stage, we would like to do the change of variables 
\[
	\Phi^{(d)}_{s_d}\circ\dotsb\circ\Phi^{(1)}_{s_1}(x)-y=\Phi^{(d)}_{t_d}\circ\dotsb\circ\Phi^{(1)}_{t_1}(x)
\]
in the last integral of \eqref{Theta circ d delta w} to estimate the increment $\|\Theta^d(w-\tau_y w)\|_{L^2_x}$. This necessitates to consider the map
\[
	H\colon\R^d\to\R^d,\quad \ttt=(t_1,\dotsc,t_d)\mapsto\Phi^{(d)}_{t_d}\circ\dotsb\circ\Phi^{(1)}_{t_1}(x_0),
\]
for a fixed $x_0\in\R^d$, and to study  the following questions:
\begin{enumerate}
	\item\label{item:diff} On what condition is $H$ locally a $C^1$-diffeomorphism?
	\item\label{item:invdiff} If realized, what estimates can we obtain on the solution $\sss \in \R^d$ to the equation
	\begin{equation}\label{eq:Fparam}
		H(\sss)=H(\ttt)+y,
	\end{equation}
	for $\ttt\in\R^d$ given and $y\in\R^d$ sufficiently small?
\end{enumerate}

\subsection{A change of variable}

\paragraph{Inversion of the map $H$.} 
We first answer the question~\ref{item:diff}. We will use an iterative procedure, so, for $i\in\{1,\dotsc,d\}$, we will denote by $H^{(i)}$ the map
\begin{equation}\label{mapFparamyi}
	H^{(i)}\colon\R^i\to\R^d,\quad \ttt_i := (t_1,\dotsc,t_i)\mapsto\Phi^{(i)}_{t_i}\circ\dotsb\circ\Phi^{(1)}_{t_1}(x_0).
\end{equation}
We fix a value $\ttt=(t_1,\dotsc,t_d)\in\R^d$ of the parameters and set
\begin{equation}\label{chainpoints}
	x_{t_1}=\Phi^{(1)}_{t_1}(x_0),\quad x_{t_2}=\Phi^{(2)}_{t_2}(x_{t_1}),\quad \dotsc,\quad x_{t_d}=\Phi^{(d)}_{t_d}(x_{t_{d-1}}).
\end{equation}  
We seek sufficient conditions under which the maps $H^{(i)}$ are locally $C^1$-diffeomorphisms. To that end, we begin by establishing some identities that will be useful later. Let $i\in\{1,\dotsc,d\}$. Notice first that
\begin{equation}\label{pull-back simplification}
   \cv^{(i)}(x_{t_i})= \frac{\di\;}{\ds} \Big|_{s=0} \Phi^{(i)}_{s} (x_{t_i})
   = \frac{\di\;}{\ds} \Big|_{s=0} \Phi^{(i)}_{t_i} (\Phi^{(i)}_{s}(x_{t_{i-1}}))
   = \di \Phi^{(i)}_{t_i}(x_{t_{i-1}})\cdot\cv^{(i)}(x_{t_{i-1}}).
\end{equation}
Moreover, let us recall that the pull-back of the vector field $\cv^{(i)}$ by a diffeomorphism $\psi\colon\R^d\to\R^d$ is given by 
\[
	\psi^\ast\cv^{(i)}(y) = \di \psi(x)^{-1}\cdot\cv^{(i)}(y),
\]
where $x=\psi^{-1}(y)$. Therefore, \eqref{pull-back simplification} gives
\[
    (\Phi^{(i)}_{t_i})^\ast  \cv^{(i)}(x_{t_i})= (\di \Phi^{(i)}_{t_i}(x_{t_{i-1}}) )^{-1}\cdot\cv^{(i)}(x_{t_i})
    = \cv^{(i)}(x_{t_{i-1}}).
\]
The concept of pull-back plays no role in itself in the following, but allows to have a geometric interpretation of the recurrence that follows, see Figure \hyperref[figure]{1}.

We can now start the induction. In the base case $i=1$, $\partial_{t_1} H^{(1)}(t_1)=\cv^{(1)}(x_{t_1})$, so $H^{(1)}$ is a $C^1$-diffeomorphism in a neighborhood of $t_1$ if 
\[
	\cv^{(1)}(x_{t_1})\not=0,
\]
which is equivalent to
\begin{equation}\label{ConditionX10}
	\cv^{(1)}(x_0)\not=0,
\end{equation}
thanks to the identity \eqref{pull-back simplification}.
If \eqref{ConditionX10} is realized, then there is a $\delta_1>0$ such that the restriction
\[
	H^{(1)}\colon V_1\to N_1:=H^{(1)}(V_1),\quad V_1:=(t_1-\delta_1,t_1+\delta_1),
\]
is $C^1$-diffeomorphism, where $N_1$ is the submanifold of $\R^d$ of dimension $1$ described by the parametrization $s_1\mapsto\Phi^{(1)}_{s_1}(x_0)$. 

Let us now consider $2\le i\le d$ and assume that for every $j = 1,\ldots,i-1$, we have found some $\delta_j>0$ such that, for $V_j=(t_j-\delta_j,t_j+\delta_j)$, we have a $C^1$-diffeomorphism
\begin{equation}\label{Hi diffeomorphism}
	H^{(i-1)}\colon V_1\times\dotsb\times V_{i-1}\to N_{i-1}:=H^{(i-1)}(V_1\times\dotsb\times V_{i-1})
\end{equation}
onto the submanifold $N_{i-1}$ of dimension $i-1$. Then, we consider the map
\begin{equation}\label{mapGi}
	G_i\colon N_{i-1}\times\R\to\R^d,\quad (x,s)\mapsto\Phi^{(i)}_{s}(x).
\end{equation}
Note that by definition of the maps $H^{(i-1)}$, $H^{(i)}$ and $G_i$, we have
\[
	H^{(i)}(\ttt_i) = G_i(H^{(i-1)}(\ttt_{i-1}),t_i),
\]
which implies
\begin{equation}\label{dHi-dGi relation}
	\di H^{(i)}(\ttt_i) 
	= \di G_i(x_{t_{i-1}},t_i)
	\begin{pmatrix} 
		\di H^{(i-1)}(\ttt_{i-1})  & 0 \\[2pt]
		0 & 1 
	\end{pmatrix}.
\end{equation}
Moreover, the differential $\di G_i(x_{t_{i-1}},t_i)$ maps $T_{(x_{t_{i-1}},t_i)}(N_{i-1}\times\R)=T_{x_{t_{i-1}}}N_{i-1}\times\R$ onto a $i$-dimensional subspace of $\R^d$, and is given by
\begin{align}
	\di G_i(x_{t_{i-1}},t_i)\cdot(v,\sigma)
	& = \di\Phi^{(i)}_{t_i}(x_{t_{i-1}})\cdot v+\sigma \cv^{(i)}(x_{t_i}) \nonumber \\
	& = \di\Phi^{(i)}_{t_i}(x_{t_{i-1}})\cdot ( v+\sigma \cv^{(i)}(x_{t_{i-1}})), \label{DGi}
\end{align}
where we used \eqref{pull-back simplification}.
It is non-injective if, for some non-trivial $(v,\sigma)\in T_{x_{t_{i-1}}}N_{i-1}\times\R$, we have 
\begin{equation}\label{N2immersion-pullback2}
	v+\sigma \cv^{(i)}(x_{t_{i-1}})=0.
\end{equation}
The equation \eqref{N2immersion-pullback2} has no non-trivial solution if 
\begin{equation}\label{ConditionX2}
	\cv^{(i)}(x_{t_{i-1}})\notin T_{x_{t_{i-1}}}N_{i-1}.
\end{equation} 
If \eqref{ConditionX2} is realized, then, in view of \eqref{dHi-dGi relation} and the induction hypothesis \eqref{Hi diffeomorphism}, there is a $\delta_i>0$ so that, considering the neighborhood $V_i=(t_i-\delta_i,t_i+\delta_i)$ of $t_i$, the restriction
\[
	H^{(i)}\colon V_1\times\dotsb\times V_{i-1}\times V_i\to N_i:=H^{(i)}(V_1\times\dotsb\times V_{i-1}\times V_i)
\]
is a $C^1$-diffeomorphism onto the submanifold $N_i$ of dimension $i$.

		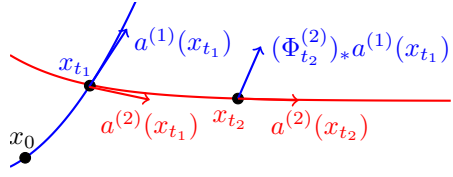
\begin{figure}\label{figure}
				\begin{center}
		\begin{tikzpicture}
			
			title={Intersection de deux courbes}, 
				title style={font=\bfseries\large, yshift=2ex},

			\clip (0.5,0.1) rectangle (7,2.3);
			
			\draw[blue, thick, domain=0.5:6, samples=100] 
			plot(\x, {\x*\x/2});
			
			\draw[red, thick, domain=0.5:6.3, samples=100] 
			plot(\x, {exp(-\x)+1});
			
			\filldraw[black] (1.55, {0.5*1.55*1.55}) circle (2pt)
			node[above left, blue, xshift=5pt] {$x_{t_1}$};
			
			\filldraw[black] (0.7, {0.5*0.7*0.7}) circle (2pt)
			node[above, xshift=-1pt] {$x_0$};
			
			\filldraw[black] (3.5, {1+exp(-3.5)}) circle (2pt)
			node[below, red, xshift=-3pt, yshift=-1pt] {$x_{t_2}$};
			
			\pgfmathsetmacro{\xone}{1.54}
			\pgfmathsetmacro{\yone}{0.5*\xone*\xone}
			\pgfmathsetmacro{\slope}{\xone} 
			
			\pgfmathsetmacro{\tvecx}{0.5}
			\pgfmathsetmacro{\tvecy}{0.5 * \slope}
			
			\draw[->,thick, blue] 
			(\xone,\yone) -- ({\xone+\tvecx},{\yone+\tvecy})
			node[midway, right, xshift=6pt, yshift=6pt] {\(\cv^{(1)}(x_{t_1})\)};
			
			\pgfmathsetmacro{\xtwo}{3.5}
			\pgfmathsetmacro{\ytwo}{1+exp(-\xtwo)}
			\pgfmathsetmacro{\slope}{1.5*\xone}
			\pgfmathsetmacro{\tvecx}{0.3}
			\pgfmathsetmacro{\tvecy}{0.3 * \slope}
			
			\draw[->,thick, blue] 
			(\xtwo,\ytwo) -- ({\xtwo+\tvecx},{\ytwo+\tvecy})
			node[right, xshift=0pt, yshift=0pt] {\((\Phi^{(2)}_{t_2})_*\cv^{(1)}(x_{t_1})\)};

			\pgfmathsetmacro{\xtwo}{3.5}
			\pgfmathsetmacro{\ytwo}{1+exp(-\xtwo)}
			\pgfmathsetmacro{\slope}{-exp(-\xtwo)}
			
			\pgfmathsetmacro{\tvecx}{0.8}
			\pgfmathsetmacro{\tvecy}{0.8 * \slope}
			
			\draw[->,thick, red] 
			(\xtwo,\ytwo) -- ({\xtwo+\tvecx},{\ytwo+\tvecy})
			node[below, xshift=8pt, yshift=0pt] {\(\cv^{(2)}(x_{t_2})\)};
			
			\pgfmathsetmacro{\xone}{1.54}
			\pgfmathsetmacro{\ytwo}{1+exp(-\xone)}
			\pgfmathsetmacro{\slope}{-exp(-\xone)}
			
			\pgfmathsetmacro{\tvecx}{0.8}
			\pgfmathsetmacro{\tvecy}{0.8 * \slope}
			
			\draw[->,thick, red] 
			(\xone,\yone) -- ({\xone+\tvecx},{\yone+\tvecy})
			node[below, xshift=0pt, yshift=0pt] {\(\cv^{(2)}(x_{t_1})\)};

		\end{tikzpicture}
			\end{center}
		\caption{The tangent plane to $N_2$ at $x_{t_2}$ is generated by the tangent vector $\cv^{(2)}(x_{t_2})$ and by the push-forward $(\Phi^{(2)}_{t_2})_{\ast}\cv^{(1)}(x_{t_1})$ of the tangent vector $\cv^{(1)}(x_{t_1})$. The pull-back of $T_{x_{t_2}}N_2$ by $\Phi^{(2)}_{t_2}$ is the plane generated by $\cv^{(2)}(x_{t_1})$ and $\cv^{(1)}(x_{t_1})$.}
		\end{figure}

Setting $N_0=\{x_0\}$, we conclude that, under the iterative conditions
\begin{equation}\label{ConditionX1d}
	\cv^{(1)}(x_0)\notin T_{x_0}N_0:=\{0\},\quad \cv^{(2)}(x_{t_1})\notin T_{x_{t_1}}N_1,\quad\dotsc,\quad \cv^{(d)}(x_{t_{d-1}})\notin T_{x_{t_{d-1}}}N_{d-1},
\end{equation}	
there is a neighborhood $V=\prod_{i=1}^d(t_i-\delta_i,t_i+\delta_i)$ of $\ttt$ such that $H\colon V\to H(V)$ is a $C^1$-diffeomorphism onto $H(V)$. It is crucial to notice that the tangent spaces $T_{x_{t_i}}N_i$ intrinsically depend on $x_0$ and $\ttt$, and do not depend on the choice of the neighborhoods $V_i$. In the following, those spaces will be denoted in the following way
\begin{equation}\label{eq:tangentspaces}
	T_{t_i} := T_{x_{t_i}}N_i,\quad i\in\{0,\ldots,d-1\},
\end{equation}
where we set $t_0 = 0$, in order to prevent any ambiguity.

\paragraph{Inversion of the map $H$: quantitative estimates.} Now, we tackle question \ref{item:invdiff} page \pageref{item:invdiff}. For that purpose, recall that we consider the space $\R^d$ with the Euclidean norm $|\cdot|$ and equip derivative matrices with the subordinate norm $\| \cdot \| $.
Assuming that \eqref{ConditionX1d} is satisfied, we seek some estimates on the solution $\sss$ to the equation 
\begin{equation}\label{from t to tau}
	H(\sss(\ttt;y))=H(\ttt)+y,
\end{equation}
for a fixed $\ttt\in\R^d$ and $y\in\R^d$ sufficiently small. Our intention is to use the change of variable $\ttt\mapsto\sss(\ttt;y)$, so we will need the two following estimates (see Proposition~\ref{prop:CVAR final} below):
\begin{equation}\label{estimates on tau1}
		|\sss-\ttt|\lesssim |y|
\end{equation}
and
\begin{equation}\label{estimates on tau2}
    \|\di\sss(\ttt)-\mathrm{Id}\|\lesssim |y|.
\end{equation}
Let us first consider the equation $H(\sss)=z$, $z$ being close to the reference value $z_*:=H(\ttt)$. To bound $|\sss-\ttt|$ in function of $|z-z_*|$, a first possible approach is to expand
\begin{align*}
    z-z_* = H(\sss)-H(\ttt)&=\int_0^1 \di H(\sss(\theta))\cdot(\sss-\ttt)\,\mathrm d\theta\nonumber\\
    & = \di H(\ttt)\cdot(\sss-\ttt)+\int_0^1 (\di H(\sss(\theta)) - \di H(\ttt))\cdot(\sss-\ttt)\,\mathrm d\theta,
\end{align*} 
where $\sss(\theta)=\theta\sss+(1-\theta)\ttt$, to get
\begin{equation}\label{preLip DH}
	\di H(\ttt)\cdot(\sss-\ttt)=z-z_*-\int_0^1 (\di H(\sss(\theta)) - \di H(\ttt))\cdot(\sss-\ttt)\,\mathrm d\theta.
\end{equation} 
This yields an estimate $|\sss-\ttt|\leq C|z-z_*|$, provided we have a bound on $\di H(\ttt)^{-1}$ and a Lipschitz estimate 
\begin{equation}\label{Lip DH}
    \|\di H(\ttt') - \di H(\ttt)\|\leq L_1|\ttt'-\ttt|,
\end{equation}
in a neighborhood of $\ttt$. We will need such a second order estimate \eqref{Lip DH} to establish the bound \eqref{estimates on tau2}, but, if we separate the question of uniqueness from the question of existence, the proof of the first bound in \eqref{estimates on tau1} necessitates only estimates on the first order quantity $\di H(\ttt)^{-1}$ (see Appendix~\ref{sec:local inversion}).

In order to use the change of variable \eqref{from t to tau}, we use Proposition~\ref{prop:from Ekeland} and Proposition~\ref{prop:from Ekeland Uniqueness} in Appendix~\ref{sec:local inversion}. For that purpose, we will need quantitative estimates on the function $H$, which requires to introduce a quantitative version of the condition \eqref{ConditionX1d} that ensures the local inversion of this function. We will use the one stated in the following definition.

\begin{definition}[Quantitative non-degeneracy] Given $x_0\in\mathbb R^d$ and $\varepsilon>0$, we say that the non-degeneracy condition holds from the point $x_0$ at some time $\ttt\in\mathbb R^d$ with precision $\varepsilon$ when
\begin{equation}\label{ConditionX1d eps}\tag{ND$_{\varepsilon,x_0,\ttt}$}
	\forall i\in\{1,\ldots,d\},\quad\dist(\cv^{(i)}(x_{t_{i-1}}),T_{t_{i-1}})>\eps,
\end{equation}
where we set $t_0 = 0$, where $x_{t_1},\ldots,x_{t_d}$ is the chain given by \eqref{chainpoints}, and where $T_{t_0},\ldots, T_{t_{d-1}}$ are the tangent spaces defined by \eqref{eq:tangentspaces}.
\end{definition}

The quantitative estimates we will need for the function $H$ and its derivatives are gathered in the following result.

\begin{proposition}[Size of $\di H(\ttt)$, $\di^2H(\ttt)$, $\di H(\ttt)^{-1}$]\label{prop:sizediff}
Let $x_0\in\R^d$, $\ttt\in\R^d$ and $\varepsilon>0$. Assume that the non-degeneracy condition \eqref{ConditionX1d eps} holds. We have then
\begin{equation}\label{sizeDH}
	\|\di H(\ttt)\|\leq 2^{d+1} e^{M|\ttt|_{\ell^1}},\quad \|\di^2H(\ttt)\|\leq 2^{2d+2} e^{2M|\ttt|_{\ell^1}},
\end{equation}
and
\begin{equation}\label{sizeDHinv}
	\|\di H(\ttt)^{-1}\|\leq \eps^{-d} (2^d-1) e^{M|\ttt|_{\ell^1}},
\end{equation}
almost surely, where we recall that $|\ttt|_{\ell^1}=|t_1|+\dotsb+|t_d|$.
\end{proposition}

\begin{proof}
We start by noticing that, for any $i=1,\dots ,d$ and $x\in\mathbb R^d$, we have
\begin{equation}\label{bound dPhi}
    \left\|\left[\di\Phi^{(i)}_{t_i}(x)\right]^{\pm1} \right\|
    \leq \exp\left( \left\|\int_0^{t_i} \di\cv^{(i)}(\Phi^{(i)}_{s}(x))\, \ds \right\| \right)
    \leq e^{M|t_i|}.
\end{equation}
Then, by \eqref{DGi}, we have for any $i =2,\dots ,d$
\[
	\|\di G_i(x_{t_{i-1}},t_i)\|
	\leq 2 e^{M|t_i|}.
\]
Therefore, \eqref{dHi-dGi relation} implies 
\[
	\|\di H^{(i)}( \ttt_i) \|
	\leq  2 e^{M|t_i|}(\|\di H^{(i-1)}(\ttt_{i-1})\| + 1 ),
\]
which gives, by iteration, that for any $i=1,\dots, d$,
\begin{equation}\label{dHi bound2}
	\|\di H^{(i)}( \ttt_i) \|
	\leq  2(2^{i}-1) e^{M|\ttt_i|_{\ell^1}}.
\end{equation}
This ends the proof of the bound on $\di H(\ttt)$ in \eqref{sizeDH}.

Now, in order to obtain the bound on $\di^2H(\ttt)$ in \eqref{sizeDH}, we write \eqref{dHi-dGi relation} in the following form, where $\uuu_i = (u_1,\dots,u_i)\in \R^i$,
\begin{equation}\label{dHi-dGi relation2}
	\di H^{(i)}(\ttt_i) \cdot \uuu_i
	= \di_x G_i(x_{t_{i-1}},t_i)\cdot(\di H^{(i-1)}(\ttt_{i-1}) \cdot \uuu_{i-1}) +  u_i \di_s G_i(x_{t_{i-1}},t_i).
\end{equation}
Considering $\vvv_i = (v_1,\ldots,v_i)\in\mathbb R^i$ and differentiating \eqref{dHi-dGi relation2}, we obtain 
\begin{align*}
	\di^2 H^{(i)}(\ttt_i)\cdot(\uuu_i, \vvv_i)
	& = \di_x^2 G_i(x_{t_{i-1}},t_i)\cdot\big(\di H^{(i-1)}(\ttt_{i-1})\cdot \uuu_{i-1}, \di H^{(i-1)}(\ttt_{i-1})\cdot \vvv_{i-1} \big)\\
	&\quad + \di^2_{s,x} G_i(x_{t_{i-1}},t_i)\cdot\big(\di H^{(i-1)}(\ttt_{i-1})\cdot \uuu_{i-1}, v_i \big)\\
	&\quad + \di^2_{s,x} G_i(x_{t_{i-1}},t_i)\cdot\big(\di H^{(i-1)}(\ttt_{i-1})\cdot \vvv_{i-1}, u_i \big)\\
	&\quad + \di_x G_i(x_{t_{i-1}},t_i)\cdot\big( \di^2 H^{(i-1)}(\ttt_{i-1})\cdot( \uuu_{i-1}, \vvv_{i-1}) \big) + u_i v_i \di_s^2 G_i(x_{t_{i-1}},t_i).
\end{align*}
Using the definition \eqref{mapGi} of $G_i$, the bound \eqref{dHi bound2}, with the formula  
\begin{equation}\label{diff2Flow}
		\di^2\Phi^{(i)}_{t_i}(x)\cdot(h,k)
		=  \int_0^{t_i} \di\Phi^{(i)}_{t_i-\sigma}(x)\cdot \big( \di^2\cv^{(i)}(\Phi_{\sigma}^{(i)}(x))\cdot \big(
		\di\Phi_{\sigma}^{(i)}(x)\cdot h,
		\di\Phi_{\sigma}^{(i)}(x)\cdot k \big)
		\big)\, \mathrm d\sigma,
\end{equation}
and 
\begin{equation}\label{Me^M sigma}
   \left| \int_0^{s} M e^{M \sigma} \di \sigma \right| \le e^{M |s|},
\end{equation}
we obtain for $|\uuu_i|,|\vvv_i|\leq 1$,
\begin{align*}
	\big\|\di^2 H^{(i)}(\ttt_i)\cdot(\uuu_i,\vvv_i) \big\| 
	& \leq   e^{2 M |t_i|} \|\di H^{(i-1)}(\ttt_{i-1}) \|^2 +  2 e^{M |t_i|} \|\di H^{(i-1)}(\ttt_{i-1}) \|\\
	& \quad + e^{M |t_i|} \big\|\di^2 H^{(i-1)}(\ttt_{i-1})\cdot(\uuu_{i-1},\vvv_{i-1}) \big\| + e^{M|t_i|} \\
	& \leq  e^{M |t_i|} \big\|\di^2 H^{(i-1)}(\ttt_{i-1})\cdot( \uuu_{i-1}, \vvv_{i-1}) \big\| + (2^{i}-1)^2 e^{2M|\ttt_i|_{\ell^1}}.
\end{align*}
The bound of $\di^2H(\ttt)$ in \eqref{sizeDH} follows by iteration. 

Finally, in order to obtain the bound \eqref{sizeDHinv} on $\|\di H(\ttt)^{-1}\|$, we assume \eqref{ConditionX1d eps} and we use \eqref{pull-back simplification} to obtain 
\begin{equation*}
	\di H^{(1)}(t_1)= \di \Phi^{(1)}_{t_1}(x_0)\cdot\cv^{(1)}(x_0),
\end{equation*}
and therefore 
\begin{equation}\label{sizeDHinv1}
	\Big\| \big[\di H^{(1)}(t_1)\big]^{-1} \Big\| \leq \varepsilon^{-1} e^{M|t_1|}.
\end{equation}
We consider now for $i=2,\dots ,d$, the equation $\di H^{(i)}(\ttt_i)\cdot \uuu_i = z$. 
Note that due to \eqref{Hi diffeomorphism},
\[
    \di H^{(i-1)}(\ttt_{i-1})\cdot\uuu_{i-1} \in T_{t_{i-1}},
\]
and \eqref{dHi-dGi relation} with  \eqref{DGi} give
\begin{equation}\label{INV dGi}
	\di H^{(i-1)}(\ttt_{i-1})\cdot \uuu_{i-1} + u_i \cv^{(i)}(x_{t_{i-1}}) = \big[\di\Phi^{(i)}_{t_i}(x_{t_{i-1}})\big]^{-1} (z).
\end{equation}
Therefore, using the property $\dist(\cv^{(i)}(x_{t_{i-1}}),T_{t_{i-1}})>\eps$ and \eqref{bound dPhi}, we obtain 
\begin{equation}\label{INV dG2 i0}
    |u_i|\leq \eps^{-1}\Big|\big[\di\Phi^{(i)}_{t_i}(x_{t_{i-1}})\big]^{-1}(z) \Big|\leq \eps^{-1}e^{M|t_i|}|z|.
\end{equation}
Inserting \eqref{INV dG2 i0} in \eqref{INV dGi}, we have thus
\begin{equation}\label{INV dG2 i}
	|\uuu_{i-1}|\leq \Big\| \big[\di H^{(i-1)}(\ttt_{i-1})  \big]^{-1} \Big\|
	2\eps^{-1}e^{M|t_i|}|z|.
\end{equation}
Combining \eqref{INV dG2 i0} with \eqref{INV dG2 i}, we get 
\[
    \Big\| \big[\di H^{(i)}(\ttt_{i})  \big]^{-1} \Big\|
	\leq  \eps^{-1}e^{M|t_i|} \Big( 2 \Big\| \big[\di H^{(i-1)}(\ttt_{i-1})  \big]^{-1} \Big\| + 1 \Big).
\]
By iteration, and using \eqref{sizeDHinv1} we obtain \eqref{sizeDHinv}.  
\end{proof}

We will also need a Lipschitz bound on the function $K$ defined in \eqref{K definition}.

\begin{proposition}[A local Lipschitz bound for $K$]
The function $K$ defined in \eqref{K definition} satisfies the following estimates for any $\ttt\in\mathbb R^d$ and $x\in\mathbb R^d$
\begin{equation}\label{K bound}
    \frac{e^{-\lambda  |\ttt|_{\ell^1}}}{(2+\lambda)^d} \leq K(\ttt,x) \leq \frac{e^{\lambda |\ttt|_{\ell^1}}}{(2-\lambda)^d},
\end{equation}	
and
\begin{equation}\label{dK bound}
	\|\di_\ttt K(\ttt,x)\| \leq C(\lambda,M, d) e^{( 2 M +\lambda )|\ttt|_{\ell^1}},
\end{equation}
almost surely.
\end{proposition}

\begin{proof}
The first bound \eqref{K bound} is a direct consequence of \eqref{E bound} and the definition \eqref{K definition}.
Using the definition \eqref{E def} of the function $E$, that we can rewrite as 
\[
 E(t,x) =  \int_0^{+\infty} e^{-2s}\, (J\Phi_{-s-t^-})(\Phi_{s+t^+}(x))\, \ds, 
\]
where $t^+:=t+t^-=\max\{0,t\}$. Using now Liouville's formula \eqref{determinantII}, we obtain
\begin{multline*}
    \partial_t E(t,x) 
    = \mathbbm{1}_{t< 0} \int_0^{+\infty} e^{-2s}\, (\divv(\cv)\circ\Phi_{t-s})(x) (J\Phi_{t-s})(\Phi_s(x))\, \ds\\
    + \mathbbm{1}_{t>0}\int_0^{+\infty} e^{-2s}(J\Phi_{-s})(\Phi_{s+t}(x))
    \bigg(\int_0^{-s} \di_x (\divv(\cv)\circ\Phi_{\sigma}) (\Phi_{s+t}(x))\cdot a(\Phi_{s+t}(x))\,\di\sigma\bigg)\,\ds,
\end{multline*}
and
\begin{multline*}
    \di_x E(t,x) = -\int_0^{+\infty} e^{-2s} (J\Phi_{-s-t^-})(\Phi_{s+t^+}(x))\\
    \times\bigg(\int_0^{s+t^-} \di_x(\divv(\cv)\circ\Phi_{-\sigma})(\Phi_{s+t^+}(x))\,\mathrm d\sigma\bigg)\circ\di \Phi_{s+t^+}(x) \, \ds.
\end{multline*}
The bound \eqref{bound dPhi} on $\di\Phi_t$ and the assumption \eqref{divergence smallness} on the divergence of $\cv$ ensure that 
\[
    |\di_x(\divv(\cv)\circ\Phi_{-\sigma})(x)|\leq Me^{M|\sigma|}\quad\text{and}\quad
    |(\divv(\cv)\circ\Phi_{-s-t^-})(x)|\leq \lambda.
\]
Thus, using also the bound from above \eqref{E bound} on $E$, the inequality \eqref{Me^M sigma}, and recalling that $0\le\lambda\le M<1/3$ by \eqref{reduction M lambda}, one can prove that for any $t \in \R$ and $x \in \R^d$, we have
\begin{equation}\label{DE}
    |\partial_t E(t,x)| \leq C(\lambda,M) e^{\lambda |t|}\quad\text{and}\quad \|\di_x E(t,x)\| \leq C(\lambda,M) e^{(M+\lambda)|t|},
\end{equation}
almost surely. Using the definition \eqref{mapFparamyi} of $H^{(i)}$, we write
\begin{equation}\label{K definition2}
	K(\ttt,x) =  \prod_{i=1}^d E\big(t_i, H^{(i-1)}(\ttt_{i-1}) \big).
\end{equation}
Using the estimate \eqref{dHi bound2} on $\di H^{(i)}$ with the pointwise bounds \eqref{E bound} on $E$ and \eqref{DE} on $\di E$, and the formula \eqref{K definition2}, we obtain \eqref{dK bound}.
\end{proof}

\paragraph{Perturbation of the non-degeneracy condition.}

The construction of the change of variable \eqref{eq:Fparam} will require perturbative results on the non-degeneracy condition \eqref{ConditionX1d eps}, all derived from the following general one.

\begin{proposition}\label{prop:perturb} Let $x_0\in\mathbb R^d$, $c>1$ and $\ttt\in\mathbb R^d$  satisfying the weak non-degeneracy condition \eqref{ConditionX1d}. Let $\eta = \eta(\ttt,c,d,M)\geq0$ satisfying
\begin{equation}\label{eq:smallnessperturb}
    2^{7d+5} c^2 \eta e^{ M(5\vert\ttt\vert_{\ell^1}+4d\eta)} \le 1.
\end{equation}
Then, for every $\sss\in B(\ttt,\eta)$, we have
\begin{equation}\label{perturbed non degeneracy}
	\forall i\in\{1,\ldots,d\},\quad\dist(\cv^{(i)}(x_{s_{i-1}}),T_{s_{i-1}})\geq\frac{c-1}{c+1}\dist(\cv^{(i)}(x_{t_{i-1}}),T_{t_{i-1}}).
\end{equation}
\end{proposition}

The proof of Proposition \ref{prop:perturb}, which is somewhat long, is postponed in Appendix \ref{app:perturbcond}. 

We can now derive two results from Proposition \ref{prop:perturb}. The first one states that the non-degeneracy condition \eqref{ConditionX1d eps} is open with respect to $\ttt\in\mathbb R^d$.

\begin{corollary}\label{cor:opencond} For every $x_0\in\mathbb R^d$ and $\varepsilon>0$, the set $\{\ttt\in\mathbb R^d : \text{\eqref{ConditionX1d eps} is satisfied}\}$ is open.
\end{corollary}

\begin{proof} Let $\ttt\in \mathbb R^d$ be such that the non-degeneracy condition \eqref{ConditionX1d eps} holds. Then, considering a constant $c\gg_{\varepsilon,\ttt}1$ large enough in \eqref{perturbed non degeneracy}, we get that the non-degeneracy condition (\hyperref[ConditionX1d eps]{$\mathrm{ND}_{\varepsilon, x_0,\sss}$}) also holds for every $\sss$ in a small neighborhood of $\ttt$.
\end{proof}

We also need the following quantitative perturbative result where we allow to have a loss in the parameter $\varepsilon>0$.

\begin{corollary}\label{cor:perturb} Let $x_0\in\mathbb R^d$, $\ttt\in\mathbb R^d$ and $\varepsilon>0$ such that the non-degeneracy condition \eqref{ConditionX1d eps} is satisfied. Consider also $\gamma>1$ and $\eta = \eta(\ttt,d,\gamma,M)>0$ satisfying
\[
	2^{7d+5} c_{\gamma}^2 \eta e^{ M(5\vert\ttt\vert_{\ell^1}+4d\eta)}\le 1,\quad\text{where}\quad 
	\text{$c_{\gamma}>\frac{\gamma+1}{\gamma-1}$ is arbitrary}.
\]
Then, for every $\sss\in B(\ttt,\eta)$, the non-degeneracy condition $\mathrm{(}$\hyperref[ConditionX1d eps]{$\mathrm{ND}_{\varepsilon/\gamma, x_0,\sss}$}$\mathrm{)}$ also holds.
\end{corollary}

\begin{proof} It suffices to apply Proposition \ref{prop:perturb} with $c=c_{\gamma}$.
\end{proof}

\paragraph{The change of variable.}
We can now properly define the change of variable \eqref{eq:Fparam} that we will use to estimate the function \eqref{Theta circ d delta w}. Precisely, given $x_0\in\mathbb R^d$, $\varepsilon>0$ and $\ray>0$, this change of variable will be performed locally on the following set
\begin{equation}\label{eq:setnd}
	\nd_{\ray,\eps,x_0} = \big\{\ttt\in B(0,\ray) : \text{\eqref{ConditionX1d eps} is satisfied}\big\}.
\end{equation}
It is fundamental to notice that thanks to Corollary \ref{cor:opencond}, the set $\nd_{\ray,\eps,x_0}$ is open.

\begin{proposition}[Change of variable]\label{prop:CVAR final}  Let $x_0\in\mathbb R^d$, $\eps>0$ and $\ray\geq1$. Fix $\omega\in\Omega$ such that the conditions
\[
	f(\omega,\cdot),\quad g(\omega,\cdot)\in L^p_x,\quad \|\cv(\omega,\cdot)\|_{C^2_b(\R^d;\R^d)}\leq M,\quad 
	\|\divv_x(\cv(\omega,\cdot))\|_{L^{\infty}(\R^d)} \leq \lambda
\]
are satisfied. Assume that the set $\nd_{\ray,\eps,x_0}$ is nonempty. Consider also $y\in\R^d$ satisfying the following smallness condition 
\begin{equation}\label{very small y}
    \eps^{-2d-\alpha} 2^{10d+12} e^{8dM(\ray+\eta)}\vert y\vert \leq 1,
\end{equation}
where $\alpha\in(0,1]$ is the one appearing in \eqref{nd-a} and $\eta = \eta(d,\ray,M)>0$ is implicitly defined by
\begin{equation}\label{eq:radius}
    2^{7d+9}\eta e^{Md(5\ray+ 4\eta)} = 1.
\end{equation}
Then, there exists a $C^2$-diffeomorphism $\sss:\ttt\in\nd_{\ray,\eps,x_0}\mapsto\sss(\ttt;y)$ satisfying
\[
	H(\sss(\ttt;y)) = H(\ttt) + y.
\]
Moreover, the following estimates
\begin{equation}\label{s close to t}
	\vert\sss(\ttt;y)-\ttt\vert\leq \eps^{-d} 2^{2d} 3 e^{dM(\ray+\eta)} \vert y\vert\le\varepsilon^\alpha
\end{equation}
and
\begin{equation}\label{size Jacobian}
	\vert\det(\di\sss(\ttt;y))-1\vert \leq \eps^{-2d} 2^{10d+11}e^{7 dM(\ray+\eta)} \vert y\vert\le\frac12 \varepsilon^\alpha
\end{equation}
are satisfied.
\end{proposition}

\begin{proof} For $\ttt\in\nd_{\ray,\eps,x_0}$ being fixed and $y\in\mathbb R^d$ small enough, we first aim to apply the implicit function theorem to the function
\[
	F : \uuu\in B(0,\eta)\mapsto F(\uuu;\ttt,y) := H(\uuu+\ttt)-H(\ttt) -y.
\]
Notice that $F(0;\ttt,0) = H(\ttt) - H(\ttt) = 0$. Moreover, it follows from the assumption \eqref{eq:radius} that
\[
    2^{7d+9}\eta e^{M(5\vert\ttt\vert_{\ell^1}+4d\eta)}\le 2^{7d+9}\eta e^{Md(5\vert\ttt\vert+4\eta)}
    \le 2^{7d+9}\eta e^{Md(5\ray+4\eta)} = 1.
\]
Corollary \ref{cor:perturb} applied with $\gamma = 2$ and $c_{\gamma} = 4$ therefore implies that for every $\vvv\in B(\ttt,\eta)$, the non-degeneracy condition $\mathrm{(}$\hyperref[ConditionX1d eps]{$\mathrm{ND}_{\varepsilon/2, x_0,\vvv}$}$\mathrm{)}$ holds. Proposition \ref{prop:sizediff} then implies that the function $F$ is differentiable on $B(0,\eta)$ and that its derivative $\di_{\uuu}F(0;\ttt,0)$ is invertible. We can therefore apply the implicit function theorem to the function $F$ (recall that the set $\nd_{\ray,\eps,x_0}$ is open by Corollary~\ref{cor:opencond}), which shows that the equation
\begin{equation}\label{eq:equation}
	F(\uuu;\ttt',y) := H(\uuu+\ttt')-H(\ttt') -y = 0,
\end{equation}
has a unique solution $\uuu = \uuu(\ttt';y)$ close to $0$ when $\vert\ttt'-\ttt\vert$ and $y$ are small enough. Moreover, the function $\uuu$ is of class $C^2$ with respect to $\ttt'$ and $y$ (since $H$ is also of class $C^2$). 

At this stage, we have no quantitative information on the function $\ttt\mapsto\uuu(\ttt;y)$, so to compensate for this, we will use Proposition \ref{prop:from Ekeland} and Proposition \ref{prop:from Ekeland Uniqueness}. By the non-degeneracy condition $\mathrm{(}$\hyperref[ConditionX1d eps]{$\mathrm{ND}_{\varepsilon/2, x_0,\vvv}$}$\mathrm{)}$, the estimates \eqref{sizeDH} and \eqref{sizeDHinv} on $\di^2 H(\vvv )$ and $\di H(\vvv )^{-1}$  respectively and the (rough) bound
\[
    \vert\vvv\vert_{\ell^1}\leq d(\ray+\eta),\quad \vvv \in B(0,\ray+\eta),
\]
the condition $\|\di F(\uuu)^{-1}\|\leq A_1$ is realized for all $\uuu\in B(0,\eta)$, with 
\[
    A_1:= \eps^{-d} 2^{2 d} e^{dM(\ray+\eta)},
\]
while the Lipschitz condition \eqref{Lip DF} is true if we set
\[
	L_1:= 2^{2d+2} e^{2 dM(\ray+\eta)}.
\]
In the notations of Proposition~\ref{prop:from Ekeland} and Proposition~\ref{prop:from Ekeland Uniqueness}, we will take 
\begin{equation}\label{eq:c1}
    A_1' := 3A_1 = \eps^{-d}2^{2 d}3e^{dM(\ray+\eta)}, \quad
	\zeta_1: = \eps^{2d} 2^{-9 d-11} e^{-7dM(\ray+\eta)},
\end{equation}
which ensures that 
\[
	2A_1<A_1',\quad A_1'\zeta_1 = \varepsilon^d2^{-7d-11}3e^{-6dM(\ray+\eta)}<\eta,\quad 
	A_1 L_1 A_1' \zeta_1 = 2^{-3d-9}3e^{-3dM(\ray + \eta)} <1.
\]
Then, for all $y\in B(0,\zeta_1)$, the equation \eqref{eq:equation} (with $\ttt'=\ttt$) has a unique solution $\uuu(\ttt;y)\in B(0,A_1'\vert y\vert)$. Setting $\sss(\ttt;y)=\uuu(\ttt;y) + \ttt$, we have constructed a map $\ttt\mapsto\sss(\ttt;y)$ from $\nd_{\ray,\eps,x_0}$ to $B(0,\ray+\eta)$ of class $C^2$, which satisfies the estimate \eqref{s close to t} by definition \eqref{eq:c1} of the constant $A_1'$. Moreover, it follows from the assumption \eqref{very small y} and the definition \eqref{eq:radius} of $\eta$ that
\[
    \eps^{-d} 2^{2 d} 3 e^{dM(\ray+\eta)}\vert y\vert = A_1'\vert y\vert < A_1' \zeta_1 < \eta.
\]
According to \eqref{s close to t} and Corollary \ref{cor:perturb} (applied with $\gamma = 2$ and $c_{\gamma} = 4$), this proves that $\sss(\ttt;y)\in\nd_{\ray+\eta,\varepsilon/2,x_0}$. Thanks to the relation $H(\sss(\ttt;y))=H(\ttt)+y$, we also obtain the following formula for the derivative of the function $\ttt\mapsto\sss(\ttt;y)$
\[
	\di\sss(\ttt;y) = \di H(\sss(\ttt;y))^{-1}\di H(\ttt).
\]
We can also estimate
\[
	\di\sss(\ttt;y)-\mathrm{Id} = \di H(\sss(\ttt;y))^{-1}(\di H(\ttt)-\di H(\sss(\ttt;y))).
\]
The bound on $\di^2H$ in \eqref{sizeDH} with \eqref{sizeDHinv}, \eqref{s close to t} and the mean value theorem give then
\[
	\|\di\sss(\ttt;y)-\mathrm{Id}\| \leq A_1 A_1' L_1 |y|\leq \zeta_1^{-1}|y|.
\]
Due to the smallness condition \eqref{very small y}, we have $\|\di\sss(\ttt;y)-\mathrm{Id}\| \leq 1$. Therefore, combining this estimate with Corollary~2.14 in \cite{IpsenRehman2008}, which states that for all matrices $A,B\in M_n(\mathbb C)$, with $B$ non-degenerate, we have
\begin{equation}\label{det estimate}
        \vert\det(A+B) - \det(B)\vert\le\bigg[\bigg(\frac{\Vert A\Vert}{\Vert B^{-1}\Vert}+1\bigg)^n-1\bigg]\vert\det(B)\vert,
\end{equation}
and the inequality
\begin{equation}\label{a^d-1}
        \forall a\in[0,2],\quad\vert a^d-1\vert \le \vert a-1\vert\sum_{k=0}^{d-1}2^k\le 2^d\vert a-1\vert,
\end{equation}
we obtain
\[
	|\det(\di\sss(\ttt;y))-1|\leq  \left(\|\di\sss(\ttt;y)-\mathrm{Id}\|+1\right)^d -1 \leq 2^d   \|\di\sss(\ttt;y)-\mathrm{Id}\| \leq 2^d \zeta_1^{-1}|y|.
\]
This is the bound \eqref{size Jacobian} on the Jacobian of $\ttt\mapsto\sss(\ttt;y)$.
\end{proof}

\begin{remark} Thanks to Corollary \ref{cor:perturb}, given any $\gamma>1$, we could also have constructed a diffeomorphism $\sss$ taking values in $\nd_{\ray+\eta_{\gamma},\varepsilon/\gamma,x_0}$, where $\eta_{\gamma}$ is implicitly defined by
\[
    2^{7d+5}c_{\gamma}^2\eta_{\gamma}e^{Md(5\ray + 4\eta_{\gamma})} = 1,\quad\text{where $c_{\gamma}>\frac{\gamma+1}{\gamma-1}$ is arbitrary}.
\]
We made the choices $\gamma = 2$ and $c_{\gamma} = 4$ arbitrarily.
\end{remark}

\subsection{Application of the change of variable}

We come back to the equation \eqref{Theta circ d delta w}, which can be written
\[
	\Theta^d(w-\tau_y w)(x) = \mathbb{E}\bigg[\int_{\R^d}(w(H(\ttt;x))-w(H(\ttt;x)-y))K(\ttt,x)\,\mathrm d\mu(\ttt)\bigg],
\]
where $\mu$ is the finite measure given by
\[
	\int_{\R^d}\varphi(\ttt)\, \mathrm d\mu(\ttt)
	= \int_{\R^d}e^{-\vert\ttt\vert_{\ell^1}} \varphi(\ttt)\,\mathrm d\ttt.
\]
Let $ \varepsilon\in(0,1)$ and $\ray\geq1$ that will be chosen later. For a given $x_0\in\R^d$, let $\nd_{\ray,\eps,x_0}$ be the set defined by \eqref{eq:setnd}. Let us consider the following new function
\[
	\Theta^d_{\ray,\eps}w(x_0) := \mathbb E\bigg[\int_{\nd_{\ray,\eps,x_0}}w(H(\ttt;x_0))K(\ttt,x_0)\,\mathrm d\mu(\ttt)\bigg].
\]
Denoting by $\nd_{\ray,\eps,x_0}^c$ the complementary set $\R^d\setminus\nd_{\ray,\eps,x_0}$, we have by the Cauchy-Schwarz inequality and the upper bound in \eqref{K bound} that
\begin{align*}
	\vert(\Theta^d w-\Theta^d_{\ray,\eps}w)(x_0)\vert^2 
	& \le\mathbb E\big[\mu(\nd_{\ray,\eps,x_0}^c)\big]
	\,\mathbb E\bigg[\int_{\mathbb R^d}\vert w(H(\ttt;x_0))K(\ttt,x_0)\vert^2\,\mathrm d\mu(\ttt)\bigg] \\
	& \le \frac{1}{(2-\lambda)^{2d}}\sup_{x\in\mathbb R^d}\mathbb E\big[\mu(\nd_{\ray,\eps,x}^c)\big]
	\,\mathbb E\bigg[\int_{\mathbb R^d}e^{2\lambda \vert\ttt\vert_{\ell^1}}\vert w(H(\ttt;x_0))\vert^2\,\mathrm d\mu(\ttt)\bigg]. 
\end{align*}
Then, by integrating with respect to the variable $x_0\in\mathbb R^d$, performing the change of variable $x_0\mapsto H(\ttt;x_0)$, whose Jacobian can be bounded by $e^{\lambda d\vert\ttt\vert_{\ell^1}}$ thanks to the estimate \eqref{determinantII}, and employing the smallness condition \eqref{reduction M lambda} on the divergence of $\cv$, we obtain the following bound
\begin{equation}\label{dist Theta to Theta main part 0}
	\big\Vert\Theta^d w-\Theta^d_{\ray,\eps}w\big\Vert_{L^2_x}\le C(\lambda,d)\big(\sup_{x\in\R^d}\mathbb E\big[\mu(\nd_{\ray,\eps,x}^c)\big]\big)^{1/2} \Vert w\Vert_{L^2_x}.
\end{equation}
We need a bound on the $\mu$-volume of the set $\nd_{\ray,\eps,x_0}^c$.

\begin{lemma}\label{lem:vol} The following estimate holds
\begin{equation}\label{bad set R eps small}
	\mathbb E\big[\mu(\nd_{\ray,\eps,x_0}^c)\big]\leq C(\cv,d)\big(\ray^{d-1}e^{-\ray}+\eps^{\alpha}\big).
\end{equation}
\end{lemma}

\begin{proof} By the definition \eqref{eq:setnd} of the set $\nd_{\ray,\eps,x_0}$ and of the condition \eqref{ConditionX1d eps}, we first have
\[
	\mathbb E\big[\mu(\nd_{\ray,\eps,x_0}^c)\big]\leq \int_{B(0,\ray)^c}e^{-\vert\ttt\vert_{\ell^1}}\,\mathrm d\ttt
	+ \sum_{i=1}^d\mathbb E\bigg[\int_{\R^d}\mathbbm{1}_{\{\dist(\cv^{(i)}(x_{t_{i-1}}),T_{t_{i-1}})<\eps \}}\,\mathrm d\mu(\ttt)\bigg],
\]
where we set $t_0 = 0$, where $x_{t_0},\ldots,x_{t_d}$ is the chain given by \eqref{chainpoints}, and where $T_{t_0},\ldots, T_{t_{d-1}}$ are the tangent spaces defined by \eqref{eq:tangentspaces}. There are therefore two terms to consider. On one hand, we roughly bound the above Gaussian integral as follows
\[
	\int_{B(0,\ray)^c}e^{-\vert\ttt\vert_{\ell^1}}\,\mathrm d\ttt\le C(d)\ray^{d-1}e^{-\ray}.
\]
On the other hand, we have the dependencies
\[
	x_{t_1}=x_{t_1}(t_1,\omega_1),\quad
	x_{t_2}=x_{t_2}(t_1,t_2,\omega_1,\omega_2),\quad\ldots,
\]
so, by independence of the vector fields $\cv^{(i)}$, and by the non-degeneracy condition \eqref{nd-a-dist}, we get that for every $i=0,\ldots,d-1$,
\[
	\mathbb E\bigg[\int_{\R^d}\mathbbm{1}_{\{\dist(\cv^{(i)}(x_{t_{i-1}}),T_{t_{i-1}})<\eps \}}\,\mathrm d\mu(\ttt)\bigg]
	= \int_{\mathbb R^d}\mathbb P(\dist(\cv^{(i)}(x_{t_{i-1}}),T_{t_{i-1}})<\eps)\,\mathrm d\mu(\ttt)
	\le C(\cv,d)\varepsilon^{\alpha}.
\]
Gathering the last three estimates leads to \eqref{bad set R eps small}.
\end{proof}

\noindent By combining the estimate \eqref{dist Theta to Theta main part 0} and Lemma \ref{lem:vol}, we therefore get that
\begin{equation}\label{dist Theta to Theta main part}
	\big\Vert\Theta^d w-\Theta^d_{\ray,\eps}w\big\Vert_{L^2_x}\leq C(\cv,d)\big(\ray^{d-1}e^{-\ray}+\eps^{\alpha}\big)^{1/2}\Vert w\Vert_{L^2_x}.
\end{equation}
Then, by using this estimate with the functions $\tau_yw$, we obtain
\begin{equation}\label{step1}
	\big\|\Theta^d(w-\tau_yw)\big\|_{L^2_x}\le C(\cv,d)\big(\ray^{d-1}e^{-\ray}+\eps^{\alpha}\big)^{1/2}\Vert w\Vert_{L^2_x}
	+ \big\|\Theta^d w-\Theta^d_{\ray,\varepsilon}\tau_yw\big\|_{L^2_x}.
\end{equation}
In the quantity
\[
	\Theta^d_{\ray,\eps}\tau_y w(x_0) = \mathbb E\bigg[\int_{\nd_{\ray,\eps,x_0}} w(H(\ttt;x_0)-y)K(\ttt,x_0)\,\mathrm d\mu(\ttt)\bigg],
\]
we can invoke Proposition \ref{prop:CVAR final} and do the change of variable $H(\sss)=H(\ttt)-y$  when $y$ satisfies the smallness condition \eqref{very small y}, to obtain
\begin{equation}\label{Theta main part tau y after cvar}
	\Theta^d_{\ray,\eps}\tau_y w(x_0) = \mathbb E\bigg[\int_{\snd_{\ray,\eps,y,x_0}} w(H(\sss;x_0))K(\ttt,x_0) \big(e^{|\sss|_{\ell^1}-|\ttt|_{\ell^1}}\big)\frac{1}{\det(\mathrm d\sss)}\,\mathrm d\mu(\sss)\bigg],
\end{equation}
where
\begin{equation}\label{modified good set}
	\snd_{\ray,\eps,y,x_0} : = \sss(\nd_{\ray,\varepsilon,x_0};y).
\end{equation}
In \eqref{Theta main part tau y after cvar}, we do a standard abuse of notations, and use the variable $\ttt$ to denote the quantity $\ttt(\sss;y)$, obtained by inversion in $\ttt$ of the relation $H(\sss)=H(\ttt)-y$. 
Our aim is now to justify (and quantify) the two following approximations:
\begin{equation}
    \Theta^d_{\ray,\eps}\tau_y w(x_0) \simeq \mathbb E\bigg[\int_{\snd_{\ray,\eps,y,x_0}}w(H(\sss;x_0))K(\sss,x_0)\,\mathrm d\mu(\sss)\bigg]=:\Theta^d_{\ray,\eps ,y} w(x_0)\label{approx Theta 1},
\end{equation}
and then
\begin{equation}
    \Theta^d_{\ray,\eps ,y} w \simeq \Theta^d w\label{approx Theta 2}.
\end{equation}
First, regarding \eqref{approx Theta 1}, we have
\begin{equation}\label{intermediate Theta}
	\big\|\Theta^d_{\ray,\eps ,y} w-\Theta^d_{\ray,\eps}\tau_y w\big\|_{L^2_x}
	\leq C(\cv,d) \eps^{-2d} e^{(\lambda d+7Md)(\ray+\eta)} \vert y\vert \Vert w\Vert_{L^2_x}.
\end{equation}
Indeed, we can apply the estimates \eqref{s close to t}, \eqref{size Jacobian}, \eqref{K bound} and \eqref{dK bound} to get
\begin{align*}
	&\ \bigg\vert K(\sss,x)  -\big(e^{|\sss|_{\ell^1}-|\ttt|_{\ell^1}}\big)\frac{K(\ttt,x)  }{\det(\di\sss)}\bigg\vert \\
	\leq &\  \frac{\vert\det(\di\sss)-1\vert K(\sss,x)+ \vert K(\sss,x) - K(\ttt,x)\vert + K(\ttt,x)\vert1- e^{|\sss|_{\ell^1}-|\ttt|_{\ell^1}}\vert}{\vert\det(\di\sss)\vert} \\
	\leq &\ C(\cv,d)\eps^{-2d} e^{(\lambda d+7Md)(\ray+\eta)} \vert y\vert,
\end{align*}
almost surely. By using Jensen's inequality for $L^2_x$-valued functions, we therefore obtain
\[
	\big\|\Theta^d_{\ray,\eps ,y} w-\Theta^d_{\ray,\eps}\tau_y w\big\|_{L^2_x}
	\le C(\cv,d) \eps^{-2d} e^{(\lambda d+7Md)(\ray+\eta)} \vert y\vert
	\,\mathbb E\bigg[\int_{\mathbb R^d}\Vert w(H(\sss;\cdot))\Vert_{L^2_x}\,\mathrm d\mu(\sss)\bigg].
\]
Performing the change of variable $x\mapsto H(\sss;x)$ in the $L^2_x$ norm, as while proving \eqref{dist Theta to Theta main part 0}, then leads to \eqref{intermediate Theta}. The justification of \eqref{approx Theta 2} is relatively similar to the proof of \eqref{dist Theta to Theta main part}. Setting $\snd_{\ray,\eps,y,x_0}^c = \R^d\setminus\snd_{\ray,\eps,y,x_0}$ and repeating the arguments used to derive \eqref{dist Theta to Theta main part 0}, we have
\[
	\big\Vert\Theta^d w-\Theta^d_{\ray,\eps,y}w\big\Vert_{L^2_x}\leq\big(\sup_{x\in\R^d}\mathbb E\big[\mu(\snd_{\ray,\eps,y,x}^c)\big]\big)^{1/2}\Vert w\Vert_{L^2_x}.
\]
We now need to estimate the expectation of the measure $\mu(\snd_{\ray,\eps,y,x_0}^c)$.

\begin{lemma}
Under the smallness assumption \eqref{very small y}, we have
\[
	\mathbb E\big[\mu(\snd_{\ray,\eps,y,x_0}^c)\big]\leq C(\cv,d) \big( \ray^{d-1}e^{-\ray}+\eps^{\alpha}\big).
\]
\end{lemma}

\begin{proof} First of all, notice that we have almost surely
\[
	\mu(\snd_{\ray,\eps,y,x_0}^c) = \mu(\mathbb R^d) - \mu(\snd_{\ray,\eps,y,x_0}) = 2^d - \mu(\snd_{\ray,\eps,y,x_0}).
\]
Moreover, by definition \eqref{modified good set} of the set $\snd_{\ray,\eps,y,x_0}$, we get that
\[
	\mu(\snd_{\ray,\eps,y,x_0}) = \mu(\sss(\nd_{\ray,\eps,x_0})) = \int_{\sss(\nd_{\ray,\eps,x_0})}\,\mathrm d\mu(\sss) = \int_{\sss(\nd_{\ray,\eps,x_0})}e^{-\vert\sss\vert_{\ell^1}}\,\mathrm d\sss.
\]
Recall from \eqref{s close to t} and \eqref{size Jacobian} that $\det(\di\sss(\ttt))\geq1-\varepsilon^{\alpha}$ and $\vert\sss-\ttt\vert\le\varepsilon^{\alpha}$. A change of variable and the (rough) bound $\vert\sss-\ttt\vert_{\ell^1}\le d\vert\sss-\ttt\vert$ therefore imply that almost surely
\[
	\mu(\snd_{\ray,\eps,y,x_0}) = \int_{\nd_{\ray,\eps,x_0}}e^{\vert\ttt\vert_{\ell^1}-\vert\sss\vert_{\ell^1}}(\det\di\sss(\ttt))\,\mathrm d\mu(\ttt)
	\geq(1-\varepsilon^{\alpha})e^{-d \varepsilon^{\alpha}}\mu(\nd_{\ray,\eps,x_0}).
\]
Finally, using the above estimates and Lemma \ref{lem:vol}, we obtain
\begin{align*}
	\mathbb E\big[\mu(\snd_{\ray,\eps,y,x_0}^c)\big] & \le 2^d - (1-\varepsilon^{\alpha})e^{-d \varepsilon^{\alpha}}\big(2^d - \mathbb E\big[\mu(\nd_{\ray,\eps,x_0}^c)\big]\big) \\
	& \le 2^d\big(1-e^{-d \varepsilon^{\alpha}} + \varepsilon^{\alpha}e^{-d \varepsilon^{\alpha}}\big) + C(\cv,d)(1-\varepsilon^{\alpha})e^{-d \varepsilon^{\alpha}}\big(\ray^{d-1}e^{-\ray}+\eps^{\alpha}\big).
\end{align*}
Since
\[ 
	1-e^{-d \varepsilon^{\alpha}} +
	\varepsilon^{\alpha}e^{-d \varepsilon^{\alpha}}\le (d+1) \varepsilon^{\alpha}\quad
	\text{and}\quad(1-\varepsilon^{\alpha})e^{-d \varepsilon^{\alpha}}\le 1,
\] 
we obtain the expected estimate.
\end{proof}

\noindent We therefore obtain
\begin{equation}\label{dist Theta to Theta hat}
	\big\|\Theta^d w-\Theta^d_{\ray,\eps,y}w\big\|_{L^2_x}
	\leq C(\cv,d) \big(\ray^{d-1}e^{-\ray}+\eps^{\alpha}\big)^{1/2}\|w\|_{L^2_x},
\end{equation}
and we can gather the various estimates \eqref{step1}, \eqref{intermediate Theta} and \eqref{dist Theta to Theta hat} to get
\[
	\sup_{|y|\le r}\|\Theta^d (w-\tau_y w)\|_{L^2_x}\leq C(\cv,d)\big(\eps^{-2d} e^{(\lambda d+7Md)(\ray+\eta)} r+ \ray^{(d-1)/2}e^{-\ray/2}+\eps^{\alpha/2} \big) \|w\|_{L^2_x}.
\]

\subsection{Optimisation}

Using the bound $\lambda\leq M$, our task is now to optimize the quantity 
\begin{equation}\label{to be optimized}
    \eps^{-2d} e^{8Md(\ray+\eta)} r+ \ray^{(d-1)/2}e^{-\ray/2}+\eps^{\alpha/2},
\end{equation}
by an accurate choice of the parameters $\eps$ and $\ray$. Simultaneously, we wish to satisfy the constraint
\begin{equation}\label{constraint eps rho}
    \eps^{-2d-\alpha} 2^{10d+12} e^{8dM(\ray+\eta)}r \leq 1,
\end{equation}
which guarantees \eqref{very small y} for $y\in B(0,r)$. Recall that $\eta$ depends on $\ray$ via the equation \eqref{eq:radius}. Taking \eqref{constraint eps rho} into account, we may revise our objectives and replace $\eps^{-2d} e^{8Md(\ray+\eta)} r$ in \eqref{to be optimized} by $\eps^{\alpha}$, and thus consider the optimization of the quantity
\begin{equation}\label{to be optimized 2}
    \ray^{(d-1)/2}e^{-\ray/2}+\eps^{\alpha/2},
\end{equation}
subject to \eqref{constraint eps rho}. Suppose that, besides \eqref{constraint eps rho}, $\ray$ satisfies $\ray\geq 1$. Then, \eqref{eq:radius} implies $\eta\leq 1\leq\ray$ and \eqref{constraint eps rho} will follow from the bound
\begin{equation}\label{constraint eps rho 2}
    \eps^{-2d-\alpha} 2^{10d+12} e^{16dM\ray}r \leq 1.
\end{equation}
Considering the behavior of the function of $(\eps,\ray)$ given by \eqref{to be optimized 2} on the domain where \eqref{constraint eps rho 2} and $\ray\geq1$ are satisfied, we may choose $(\eps,\ray)$ such that \eqref{constraint eps rho 2} is saturated. Then there remains to consider the expression
\[
    \ray^{(d-1)/2}e^{-\ray/2}+e^{\frac{8\alpha d M}{2d+\alpha}\ray}r^{\frac{\alpha}{4d+2\alpha}}.
\]
We choose $\ray$ to ensure the balance $e^{-\ray/2}=e^{\frac{8\alpha d M}{2d+\alpha}\ray}r^{\frac{\alpha}{4d+2\alpha}}$, that is to say (assuming $r<1$)
\begin{equation}\label{ray function of r}
    \ray=\frac{\alpha}{2d+\alpha+16\alpha d M}|\ln(r)|.
\end{equation}
This leads to the bound
\begin{equation}\label{key estim Theta pre2}
	\sup_{|y|\le r}\|\Theta^d (w-\tau_y w)\|_{L^2_x}\leq C(a,d) r^{2d\bar\theta}|\ln(r)|^{(d-1)/2} \|w\|_{L^2_x},
\end{equation}
with
\begin{equation}\label{bar theta}
    \bar\theta(\alpha,d,M)=\frac{1}{4d}\frac{\alpha}{2d+\alpha+16\alpha d M}.
\end{equation}
Since $M\leq 1$, the radius $\ray$ given by \eqref{ray function of r} satisfies $\ray\geq 1$ if 
\[
    r<r_0(\alpha,d):=\exp\left(-\frac{\alpha}{2d+\alpha+16\alpha d}\right).
\]
Let $\theta$ be given such that
\[
    \theta<\delta(2,\alpha,d)=\frac1{4d}\frac{1}{1+\frac{2d}{\alpha}}.
\]
We invoke again (see the final step of Section~\ref{subsec:some reductions}) the use of the transformation \eqref{eq:modified eq} with $\Lambda$ large (and we recall again that this transformation modifies the value of the constant $C$ in the non-degeneracy condition \eqref{nd-a}, so influences directly the value of the constant $C(\cv,d)$ in \eqref{key estim Theta pre2}). More precisely, we choose $\Lambda=\Lambda(M,\theta)$ to ensure 
\begin{equation}\label{theta VS bar theta}
    \theta<\bar\theta<\delta(2,\alpha,d),
\end{equation}
with $\bar\theta$ defined in \eqref{bar theta} (note well that, in \eqref{theta VS bar theta}, $\bar\theta$ now depends on $d$, $\alpha$ and $\theta$). We then use the fact that the quantity
\[
\sup_{r<r_0(\alpha,d)} r^{2d(\bar\theta-\theta)}|\ln(r)|^{(d-1)/2}
\]
is finite, and depends on $\alpha$, $d$, $\theta$ to deduce from \eqref{eq:intereq}, \eqref{theta by higher powers}, \eqref{key estim Theta 2} and \eqref{key estim Theta pre2} the bound 
\[
	\forall r\in(0,r_0], \quad\sup_{|y| \le r}\|\tau_y\varrho-\varrho\|_{L^2_x}\leq \kappa \|h\|_{L^2_{\omega,x}} r^{\theta},
\]
as displayed in \eqref{AV-lemII} (recall that the function $h$ is given by \eqref{eq:functionh}), with a constant $\kappa$ depending on $\Vert\divv_x(\cv)\Vert_{L^{\infty}}$, $M$, $C$, $\alpha$, $p$, $\theta$ and $d$.

\appendix

\section{The stationary transport equation}\label{sec:stationary transport}

In this first appendix, we study the existence and uniqueness of solutions of the following deterministic stationary transport equation
\begin{equation}\label{stationary transport c}
	f(x)-\cv(x)\cdot\nabla_xf(x) = g(x),\quad x\in\mathbb R^d,
\end{equation}
where $\cv\in C^1_b(\R^d;\R^d)$ is a fixed deterministic vector field.

\subsection{Strong solutions}

First of all, we are interested in the existence and uniqueness of strong solutions of the equation \eqref{stationary transport c}, \textit{i.e.} to the functions $f\in C^1(\mathbb R^d)$ satisfying \eqref{stationary transport c} for every $x\in \mathbb R^d$.

\begin{proposition}[Uniqueness]\label{prop:strongsol} Let $p\in[1,+\infty]$ and $g\in L^p(\mathbb R^d)\cap C^0(\mathbb R^d)$. Assume that the following smallness condition holds
\begin{equation}\label{eq:smallstronguniq}
	\Vert\divv(\cv)\Vert_{L^{\infty}(\mathbb R^d)}<p.
\end{equation}
Then, if the function $f\in L^p(\mathbb R^d)\cap C^1(\mathbb R^d)$ is a strong solution of the equation \eqref{stationary transport c}, then $f$ is necessarily given by the following formula
\begin{equation}\label{eq:resolventestimate}
	f = \int_0^\infty e^{-t}g\circ\Phi_t\, \di t,
\end{equation}
where $(\Phi_t)_{t\in\mathbb R}$ stands for the flow associated to the vector field $\cv$. Moreover, the function $f$ satisfies the following bound
\begin{equation}\label{eq:contraction}
	\|f\|_{L^p(\R^d)}\leq \frac{p}{p-\Vert\divv(a)\Vert_{L^{\infty}(\mathbb R^d)}}\|g\|_{L^p(\R^d)}.
\end{equation}
\end{proposition}

\begin{proof} Let $g\in L^p(\mathbb R^d)\cap C^0(\mathbb R^d)$ and assume that $f\in L^p(\mathbb R^d)\cap C^1(\mathbb R^d)$ is solution of the transport equation \eqref{stationary transport c}. We divide the proof of Proposition \ref{prop:strongsol} in three steps.

\medskip

\noindent\textbf{$\triangleright$ Step 1.1: Computations.} We aim at obtaining a first formula relating $f$ and $g$. By composing \eqref{stationary transport c} with the flow $\Phi_{-t}$, we obtain
\[
	f\circ\Phi_{-t}+\frac{\di\ }{\di t}(f\circ\Phi_{-t}) = g\circ\Phi_{-t},
\]
and thus
\[
	\frac{\di\ }{\di t}(e^t f\circ\Phi_{-t})= e^t g\circ\Phi_{-t}.
\]
We then integrate between $0$ and $t$ to get the following formula
\[
	e^tf\circ\Phi_{-t}-f= \int_0^t e^sg\circ\Phi_{-s} \, \di s,
\]
and consequently
\[
	 f\circ\Phi_{-t}=e^{-t}f+ \int_0^t e^{-(t-s)}g\circ\Phi_{-s} \, \di s.
\]
In order to later obtain a formula involving the function $f$ alone, we compose with $\Phi_t$
\begin{equation}\label{f with t}
	f = e^{-t} f\circ\Phi_t+ \int_0^t e^{-(t-s)}g\circ\Phi_{t-s} \, \di s
	= e^{-t} f\circ\Phi_t+ \int_0^t e^{-s}g\circ\Phi_s \, \di s.
\end{equation}
Notice that the above computations only require $g\in C^0(\mathbb R^d)$ and $f\in C^1(\mathbb R^d)$.

\medskip

\noindent\textbf{$\triangleright$ Step 1.2: Estimate in $L^p$.} We can now prove the bound \eqref{eq:contraction}. To that end, consider a radius $r>0$ and a time $t>0$. The strategy is to establish estimates in $L^p(B(0,r))$ and pass to the limit as $r\rightarrow+\infty$. In order to simplify the writing, we set $\lambda = \Vert\divv(\cv)\Vert_{L^{\infty}}$ and $M_0 = \Vert\cv\Vert_{L^{\infty}}$ in the following. Notice that 
\[
	\Phi_t(B(0,r))\subset B(0,r+M_0t),
\]
since we have for every $x\in\mathbb R^d$
\[
	|\Phi_t(x)|=\left|x+\int_0^t\cv\circ\Phi_s(x)\, \di s\right|\leq|x|+M_0 t.
\]
It therefore follows, after the change of variables $y=\Phi_t(x)$, and by Liouville's formula
\begin{equation}\label{determinantII}
	J\Phi_t(x) := \det (\mathrm{d} \Phi_t (x)) = \exp\left(\int_0^t \divv (\cv) (\Phi_s(x))\,\mathrm ds \right) \leq \exp(\lambda|t|),
\end{equation}
that when $p<+\infty$, 
\begin{equation}\label{bound Lq cv}
	\|g\circ\Phi_t\|_{L^p(B(0,r))}=\left(\int_{\Phi_t(B(0,r))} |g(y)|^p \det (\mathrm{d} \Phi_{-t}(y)) \, \di y\right)^{1/p}\leq e^{\lambda t/p} \|g\|_{L^p(B(0,r+M_0 t))}.
\end{equation}
Clearly, this estimate holds true for $p=+\infty$. We then apply the bound \eqref{bound Lq cv} to \eqref{f with t} to obtain
\begin{align*}
	\|f\|_{L^p(B(0,r))}
	& \leq e^{-t} \|f\|_{L^p(B(0,r+M_0 t))} + \int_0^t e^{-(1-\lambda/p)s} \|g\|_{L^p(\R^d)} \, \di s \\
	& \leq e^{-t} \|f\|_{L^p(\R^d)} + \int_0^t e^{-(1-\lambda/p)s} \|g\|_{L^p(\R^d)} \, \di s.
\end{align*}
Therefore, by letting $t\to+\infty$ , then $r\to+\infty$, while using the smallness assumption \eqref{eq:smallstronguniq}, we obtain the expected estimate \eqref{eq:contraction}.

\medskip

\noindent\textbf{$\triangleright$ Step 1.3: Formula.} It now only remains to establish the formula \eqref{eq:resolventestimate}. Actually, we only need to justify that we can let $t\to+\infty$ in \eqref{f with t}. To that end, set $F$ the right-hand side of \eqref{eq:resolventestimate}. It is well-defined pointwize, and continuous by continuity under the integral sign. On the basis of \eqref{f with t}, we have
\[
	f-F= e^{-t}f\circ\Phi_t- \int_t^\infty e^{-s} g\circ\Phi_s \, \di s.
\]
It follows that for every $r>0$, 
\[
	\|f-F\|_{L^p(B(0,r))}\leq e^{-(1-\lambda/p)t} \|f\|_{L^p(\R^d)}
	+\int_t^\infty e^{-(1-\lambda/p)s} \|g\|_{L^p(\R^d)}\, \di s,
\]
and thus, as can be seen by letting $t\to+\infty$ while using the condition \eqref{eq:smallstronguniq}, we get $f=F$ on $B(0,r)$. Since $r$ is arbitrary, we obtain \eqref{eq:resolventestimate}, as an equality pointwize and also in $L^p(\mathbb R^d)$.
\end{proof}

The following result on the existence of smooth solutions will not be used elsewhere in this work, but is presented and proven for completeness of the present study.

\begin{proposition}[Existence] Let $p\in[1,+\infty]$ and $g\in L^p(\mathbb R^d)\cap C^1_b(\mathbb R^d)$. Assume that the smallness condition \eqref{eq:smallstronguniq} holds. Then, the function $f\in L^p(\mathbb R^d)\cap C^1_b(\mathbb R^d)$ defined by \eqref{eq:resolventestimate} is a strong solution of the equation \eqref{stationary transport c} and satisfies the bound \eqref{eq:contraction}.
\end{proposition}

\begin{proof} Let us introduce the notation $M_1 = \|\di\cv\|_{L^{\infty}}$. We first justify that  $f$ is of class $C^1$ on $\mathbb R^d$. Since
\[
	\di\Phi_t(x)=\exp\bigg(\int_0^t\di\cv\circ\Phi_s(x)\,\mathrm ds \bigg),\quad\text{and thus}\quad\|\di\Phi_t(x)\|\leq e^{M_1t},
\]
the fact that $f\in C^1(\mathbb R^d)$ is justified by derivation under the integral sign. We can then justify \eqref{stationary transport c} by direct differentiation in \eqref{eq:resolventestimate}, but this is clumsy. Since $\cv\cdot\nabla_x$ is the generator associated to the group $(\Phi_t)_{t\in\mathbb R}$, it is better to examine the quantity $f\circ\Phi_\tau-f$ for a small $\tau$. Since
\[
	f\circ\Phi_\tau= \int_0^\infty e^{-t}g\circ\Phi_{t+\tau}\,\mathrm dt,
\]
by the group property of $(\Phi_t)_{t\in\mathbb R}$, the change of variable $t'=t+\tau$ gives
\[
	f\circ\Phi_\tau = e^{\tau}\int_\tau^\infty e^{-t} g\circ\Phi_t\,\mathrm dt.
\]
We therefore obtain
\begin{equation}\label{ftau f}
	f\circ\Phi_\tau-f = \int_0^\infty (e^{\tau-t} - e^{-t})g\circ\Phi_t\,\mathrm dt
	-\int_0^\tau e^{\tau-t}g\circ\Phi_t\,\mathrm dt.
\end{equation}
The last term in the right-hand side of \eqref{ftau f} is
\[
	-\tau \int_0^1 e^{\tau-\tau t}g\circ\Phi_{\tau t}\,\mathrm dt,
\]
and thus behaves (pointwise) as $-\tau g+o(\tau)$. The first term in the right-hand side of \eqref{ftau f} can be rewritten as
\[
	\int_0^\infty (e^{\tau}-1)e^{-t}g\circ\Phi_t\,\mathrm dt,
\]
and is thus $\tau f+o(\tau)$. We therefore obtain that for every $x\in\mathbb R^d$,
\[
	\cv(x)\cdot\nabla_x f(x) = \lim_{\tau\rightarrow0}\frac{f\circ\Phi_\tau(x)-f(x)}{\tau} = -g(x) + f(x),
\]
as desired.
\end{proof}

\subsection{Weak solutions}

In this section, we investigate the existence and uniqueness of weak solutions in $L^p(\mathbb R^d)$ for the transport equation \eqref{stationary transport c}, defined as follows.

\begin{definition}[Weak solution] Let $p\in[1,+\infty]$ and let $g\in L^p(\R^d)$. A function $f\in L^p(\R^d)$ is said to be a weak solution to the stationary transport equation \eqref{stationary transport c} when it satisfies for every $\varphi\in\mathcal{D}(\R^d)$
\[
    \int_{\R^d}f(x)(\varphi(x)+\divv(\cv(x)\varphi(x)))\,\mathrm dx = \int_{\mathbb R^d}g(x)\varphi(x)\,\mathrm dx.
\]
\end{definition}

Our result on weak solutions for the equation \eqref{stationary transport c} is the following

\begin{proposition}\label{prop:weaksol} Let $p\in[1,+\infty]$ and $g\in L^p(\mathbb R^d)$. Assume that the smallness condition \eqref{eq:smallstronguniq} holds. Then, the function $f\in L^p(\mathbb R^d)$ defined by \eqref{eq:resolventestimate} is the only weak solution in $L^p(\mathbb R^d)$ of the transport equation \eqref{stationary transport c} in $\mathcal D'(\mathbb R^d)$. Moreover, the function $f$ still satisfies the bound \eqref{eq:contraction}.
\end{proposition}

\begin{proof} In order to simplify the writing, we set $\lambda = \Vert\divv(a)\Vert_{L^{\infty}}$. 

\medskip

\noindent\textbf{Existence.} Let us consider the operator $S: L^p(\mathbb R^d)\rightarrow L^p(\mathbb R^d)$ defined for every $g\in L^p(\mathbb R^d)$ by
\[
	Sg = \int_0^\infty e^{-t}g\circ\Phi_t\, \di t.
\]
We check that the operator $S$ is well-defined. For $g\in L^p(\R^d)$, we apply the change of variable $y=\Phi_t(x)$, whose Jacobian determinant is given by Liouville's formula \eqref{determinantII}, to obtain by Jensen's inequality for $L^p$ functions
\begin{align*}
	\Vert Sg\Vert_{L^p} 
	&\leq  \int_0^\infty e^{-t}  \Vert g\circ\Phi_t\Vert_{L^p}\, \dt  \\
	&\leq \int_0^\infty e^{-(1-\lambda/p)t} \Vert g\Vert_{L^p}\,\dt
	= \frac{p}{p - \lambda}\Vert g\Vert_{L^p}.
\end{align*}
Moreover, if $\sigma\in\R$, the group property of $(\Phi_t)_{t\in\mathbb R}$ gives
\[
Sg\circ\Phi_\sigma=\int_0^\infty e^{-t} g\circ\Phi_{t+\sigma}\,\mathrm dt
= e^\sigma\bigg(Sg-\int_0^\sigma e^{-t}g\circ\Phi_t\,\mathrm dt\bigg).
\]
We test this identity against $\varphi\in\mathcal{D}(\R^d)$. Fubini's theorem and the change of variable $y=\Phi_t(x)$ (inverted in $x=\Phi_{-t}(y)$ since $\cv$ is an autonomous vector field) yield
\[
\dual{Sg}{(\varphi\circ\Phi_{-\sigma})(J\Phi_{-\sigma})}_{\mathcal D',\mathcal D} =
e^\sigma\bigg(\dual{Sg}{\varphi}_{\mathcal D',\mathcal D}
-\int_0^\sigma e^{-t}\dual{g}{(\varphi\circ\Phi_{-t})(J\Phi_{-t})}_{\mathcal D',\mathcal D}\,\mathrm dt\bigg).
\]
We evaluate the derivative with respect to $\sigma$ of this last expression at $\sigma=0$ to obtain
\[
\dual{Sg}{\varphi+\divv_x(\cv\varphi)}_{\mathcal D',\mathcal D} = \dual{g}{\varphi}_{\mathcal D',\mathcal D}.
\]
This proves that $Sg$ satisfies \eqref{stationary transport c} in $\mathcal{D}'(\R^d)$, as expected.

\medskip

\noindent\textbf{Uniqueness.} Let us now consider a function $f\in L^p(\R^d)$ which satisfies the equation \eqref{stationary transport c} in $\mathcal{D}'(\R^d)$ with $g=0$. In order to prove that $f= 0$, we use a regularization procedure. Let $(\rho_\eps)_{\eps>0}$ be an approximation of the identity. By assumption, the function $f$ satisfies for every $\varphi\in\mathcal{D}(\R^d)$,
\begin{equation}\label{weak stationary transport II}
	\int_{\R^d}f(x)(\varphi(x)+\divv(\cv(x)\varphi(x)) )\,\mathrm dx=0.
\end{equation}
Pick then $\varphi\in\mathcal{D}(\R^d)$ and use $x\in\mathbb R^d\mapsto\varphi(x)\rho_\eps(y-x)$ as test functions in \eqref{weak stationary transport II} for every $y\in\mathbb R^d$ to get the following equation in $\mathcal D'(\mathbb R^d)$
\begin{equation}\label{f regularized}
	f^\eps-\cv\cdot\nabla_x f^\eps=r_\eps\quad\text{with}\quad r_\eps:=(\cv\cdot\nabla_x f)^\eps-\cv\cdot\nabla_x f^\eps,
\end{equation}
where we set generically $\psi^\eps:=\psi\ast\rho_\eps$. Since the smallness condition \eqref{eq:smallstronguniq} holds by assumption, and that $r_{\varepsilon}\in L^p(\mathbb R^d)\cap C^0(\mathbb R^d)$ by the equation \eqref{f regularized}, the left-hand side $f^\eps-\cv\cdot\nabla_x f^\eps$ itself belonging to this space, Proposition \ref{prop:strongsol} implies that the function $f^{\varepsilon}\in L^p(\mathbb R^d)\cap C^1(\mathbb R^d)$ satisfies the bound
\[
	\|f^\eps\|_{L^p}\leq \frac{p}{p - \lambda}\|r_\eps\|_{L^p}. 
\]
The commutation lemma \cite[Lemma II.1]{DiPernaLions89} then gives $r_\eps\to 0$ as $\varepsilon\rightarrow0^+$ in $L^p(\mathbb R^d)$, and thus $f=0$ as expected.
\end{proof}

\section{Perturbation of the non-degeneracy condition}\label{app:perturbcond}

In this second appendix, we give the proof of Proposition \ref{prop:perturb}. Let us recall that given $x_0\in\mathbb R^d$, $c>1$, $\ttt\in\mathbb R^d$ satisfying the weak non-degeneracy condition \eqref{ConditionX1d}, and $\eta = \eta(\ttt,c,d,M)\geq0$ satisfying the smallness condition \eqref{eq:smallnessperturb}, we aim at proving that for every $\sss\in B(\ttt,\eta)$
\[
	\forall i\in\{1,\ldots,d\},\quad\dist(\cv^{(i)}(x_{s_{i-1}}),T_{s_{i-1}})\geq\frac{c-1}{c+1}\dist(\cv^{(i)}(x_{t_{i-1}}),T_{t_{i-1}}).
\]
Introducing the following general notations, where $u_i$ stands for $t_i$ or $s_i$ in all the proof,
\[
	d_{u_i} := \dist(\cv^{(i+1)}(x_{u_i}),T_{u_i}),
\]
the above estimate will be deduced from the following one
\begin{equation}\label{eq:interin}
	\vert d_{t_i}-d_{s_i}\vert \leq \frac1c (d_{t_i}+d_{s_i}).
\end{equation}
The proof is divided in several steps.
\newline

\noindent\textbf{$\triangleright$ Prerequisite: Gram matrix.} We will use the notion of Gram matrix, whose definition is the following. Given a vector space $E$ endowed with a scalar product $\langle\cdot,\cdot\rangle$, and $\alpha_1,\ldots,\alpha_n$ some vectors in $E$, the Gram matrix $G(\alpha_1,\ldots,\alpha_n)$ of $(\alpha_1,\ldots,\alpha_n)$ is defined by
\begin{equation}\label{eq:grammatrix}
	G(\alpha_1,\ldots,\alpha_n) = (\alpha_1,\ldots,\alpha_n)^T(\alpha_1,\ldots,\alpha_n) = (\langle\alpha_i,\alpha_j\rangle)_{1\le i,j\le n}.
\end{equation}
As stated e.g. in \cite[p.130]{Deutsch2001}, if $x\in E$ and $V\subset E$ is a finite-dimensional vector subspace, say of dimension $n$, then the distance from $x$ to $V$ is given by the following formula
\begin{equation}\label{dist via Gram}
	\dist(x,V) = \bigg[\frac{\det G(x,\alpha_1,\ldots,\alpha_n)}{\det G(\alpha_1,\ldots,\alpha_n)}\bigg]^{1/2},
\end{equation}
where $(\alpha_1,\ldots,\alpha_n)$ stands for any basis of $V$. It is precisely by means of this formula that we will be able to establish \eqref{eq:interin}.
\newline

\noindent\textbf{$\triangleright$ Step 1: Basis of the tangent spaces}. We need to construct basis for the tangent spaces $T_{u_i}$ defined by \eqref{eq:tangentspaces}, in order to be able to use \eqref{dist via Gram}. This can be achieved in a generic way as follows: if we initialize $\alpha_1^{u_1}=\cv^{(1)}(x_0)$, then we can successively deduce a basis $(\alpha_1^{u_{i+1}},\ldots,\alpha_{i+1}^{u_{i+1}})$ of $T_{u_{i+1}}$ from a basis $(\alpha_1^{u_i},\ldots,\alpha_i^{u_i})$ of $T_{u_i}$, by setting (recall \eqref{DGi})
\begin{align}
	& \alpha_k^{u_{i+1}} = (\Phi^{(i+1)}_{u_{i+1}})_{\ast}\alpha_k^{u_i} = \di\Phi^{(i+1)}_{u_{i+1}}(x_{u_i})\cdot \alpha_k^{u_i},\quad 1\leq k\leq i, \label{eq:basisI} \\
	& \alpha_{i+1}^{u_{i+1}} = (\Phi^{(i+1)}_{t_{i+1}})_{\ast}\cv^{(i+1)}(x_{u_i}) = \cv^{(i+1)}(x_{u_{i+1}}). \nonumber
\end{align}
Let us introduce the general notations
\begin{equation}\label{eq:defxb}
	X_{u_i} = (\cv^{(i+1)}(x_{u_i}), \alpha_1^{u_i},\ldots,\alpha_i^{u_i})\quad\text{and}\quad
	\mathcal B_{u_i} = (\alpha_1^{u_i},\ldots,\alpha_i^{u_i}),
\end{equation}
Using the formula \eqref{dist via Gram} (twice) and the inequality $\vert a-b\vert^2 \leq |a^2-b^2|$ (for $a,b\geq 0$), we get  
\begin{align}\nonumber
	\vert d_{t_i} - d_{s_i}\vert^2 & \leq \vert d_{t_i}^2 - d_{s_i}^2\vert = \bigg\vert\frac{\det G(X_{t_i})}{\det G(\mathcal B_{t_i})} 
	- \frac{\det G(X_{s_i})}{\det G(\mathcal B_{s_i})}\bigg\vert \\[2pt]
	& \leq \frac{(\det G(X_{s_i}))\vert\det G(\mathcal B_{s_i}) - \det G(\mathcal B_{t_i})\vert}{(\det G(\mathcal B_{s_i}))(\det G(\mathcal B_{t_i}))} 
	+  \frac{\vert\det G(X_{s_i}) - \det G(X_{t_i})\vert}{\det G(\mathcal B_{t_i})} \nonumber \\[2pt]
	& = \frac{d_{s_i}^2\vert\det G(\mathcal B_{s_i}) - \det G(\mathcal B_{t_i})\vert}{\det G(\mathcal B_{t_i})} 
	+ \frac{\vert\det G(X_{s_i}) - \det G(X_{t_i})\vert}{\det G(\mathcal B_{t_i})}. \label{eq:firstestimate}
\end{align}
It appears that we need bounds on the (difference of) determinants of Gram matrices. This is the purpose of the following step.
\newline

\noindent \textbf{$\triangleright$ Step 2: Estimates on Gram determinants.} Let us return temporarily to a general discussion on Gram matrices. Let us consider a positive integer $n\geq1$ and $X,Y\in (\mathbb R^d)^n$. Assume that the matrix $G(Y)$ is invertible (this happens when the column vectors of $G(Y)$ are linearly independent). we will use the following notations
\[
    \vert X\vert = (\vert X_1\vert^2 + \cdots + \vert X_n\vert^2)^{1/2},\quad 
    \vert Y\vert = (\vert Y_1\vert^2 + \cdots + \vert Y_n\vert^2)^{1/2}.
\]
We first notice from the estimate \eqref{det estimate} and the inequality $1 \leq \Vert A^{-1}\Vert \Vert A\Vert$, which holds for every non-degenerate matrix $A\in M_n(\mathbb R)$, that the difference between the Gram determinants $\det G(X)$ and $\det G(Y)$ can be bounded as follows
\begin{align*}
	\vert\det G(X) - \det G(Y)\vert & \le \bigg[\bigg(\frac{\Vert G(X) - G(Y)\Vert}{\Vert G(Y)^{-1}\Vert}+1\bigg)^n-1\bigg]\det G(Y) \\
	& \le \big[\big(\Vert G(Y)\Vert\Vert G(X) - G(Y)\Vert+1\big)^n-1\big]\det G(Y).
\end{align*}
Moreover, under the smallness assumption $\Vert G(Y)\Vert\Vert G(X) - G(Y)\Vert\le 1$, the estimate \eqref{a^d-1} implies
\[
	\vert\det G(X) - \det G(Y)\vert\le 2^n\Vert G(Y)\Vert\Vert G(X) - G(Y)\Vert\det G(Y).
\] 
Notice then from the definition \eqref{eq:grammatrix} of the Gram matrix that $\Vert G(Y)\Vert\le\vert Y\vert^2$ and also
\[
	G(X) - G(Y) = X^TX - Y^TY = (X-Y)^TX + Y^T(X-Y),
\]
which implies
\[
	\Vert G(X) - G(Y)\Vert\le\vert X-Y\vert(\vert X\vert + \vert Y\vert).
\]
As a consequence,
\begin{equation}\label{eq:smalldet}
	\vert Y\vert^2\vert X-Y\vert(\vert X\vert + \vert Y\vert)\le 1\quad\Longrightarrow\quad\Vert G(Y)\Vert\Vert G(X) - G(Y)\Vert\le 1,
\end{equation}
and under this assumption, we have 
\[
	\vert\det G(X) - \det G(Y)\vert\le 2^n\vert Y\vert^2\vert X-Y\vert(\vert X\vert +\vert Y\vert)\det G(Y).
\]

Let us now come back to the distances $d_{u_i}$ that we aim studying. Notice that since $\mathcal B_{u_i}$ is a basis of the tangent $T_{u_i}$, the matrix $G(\mathcal B_{u_i})$ is invertible. Moreover, thanks to the the weak non-degeneracy condition \eqref{ConditionX1d}, the family $(\cv^{(i+1)}(x_{u_i}), \alpha_1^{u_i},\ldots,\alpha_i^{u_i})$ is linearly independent, so the matrix $G(X_{t_i})$ is invertible as well. Therefore, when the smallness assumption \eqref{eq:smalldet} holds for $(X,Y) = (X_{s_i},X_{t_i})$ and $(X,Y) = (\mathcal B_{s_i},\mathcal B_{t_i})$, we deduce from  the formula \eqref{dist via Gram}, the inequality \eqref{eq:firstestimate} and the above estimate (with $n=i$ and $n=i+1$) that
\begin{align}
	\vert d_{t_i} - d_{s_i}\vert^2 & \le \frac{2^i d_{s_i}^2\vert\mathcal B_{t_i}\vert^2\vert\mathcal B_{t_i}-\mathcal B_{s_i}\vert(\vert\mathcal B_{t_i}\vert +\vert\mathcal B_{s_i}\vert)\det G(\mathcal B_{t_i})}{\det G(\mathcal B_{t_i})}  \nonumber \\
	& \hspace{138pt} + \frac{2^{i+1}\vert X_{t_i}\vert^2\vert X_{t_i}-X_{s_i}\vert(\vert X_{t_i}\vert +\vert X_{s_i}\vert)\det G(X_{t_i})}{\det G(\mathcal B_{t_i})} \nonumber \\[5pt]
	& \hspace{-2mm} = 2^id_{s_i}^2\vert\mathcal B_{t_i}\vert^2\vert\mathcal B_{t_i}-\mathcal B_{s_i}\vert(\vert\mathcal B_{t_i}\vert + \vert\mathcal B_{s_i}\vert) 
	+ 2^{i+1} d_{t_i}^2\vert X_{t_i}\vert^2\vert X_{t_i}-X_{s_i}\vert(\vert X_{t_i}\vert +\vert X_{s_i}\vert). \label{eq:avantder}
\end{align}
It now remains to check that the condition \eqref{eq:smalldet} actually holds, and to estimate the norms involved in the above estimates.
\newline

\noindent \textbf{$\triangleright$ Step 3: Estimates on the tangent spaces.} In view of the definition \eqref{eq:defxb} of the vectors $X_{u_i}$ and $\mathcal B_{u_i}$, we need to derive estimates for the vectors $\alpha^{u_i}_k$ (and their differences). Notice first that by the bound \eqref{dHi bound2} on $\di H^{(i)}(\uuu_i)$ and the fact that $\sss\in B(\ttt,\eta)$,
\[
	|\uuu|_{\ell^1} \leq |\ttt|_{\ell^1} + d \eta, \quad \forall \uuu \in B(\ttt,\eta),
\]
we have by the mean value theorem
\begin{equation}\label{xi VS zi}
	\vert x_{t_i}-x_{s_i}\vert = \vert H^{(i)}(\ttt_i)-H^{(i)}(\sss_i)\vert
	\le\sup_{\textbf{u}_i\in[\ttt_i,\sss_i]}\Vert\di H^{(i)}(\mathbf{u}_i)\Vert\vert\ttt_i-\sss_i\vert
	\le 2^{d+1}\eta e^{M(\vert\ttt_i\vert_{\ell^1}+d\eta)}.
\end{equation}
Using the definition \eqref{eq:basisI} of the vectors $\alpha^{u_{i+1}}_k$ and the bound \eqref{bound dPhi}, we obtain for every $\uuu\in\mathbb R^d$,
\begin{equation}\label{bound11}
	\vert \alpha^{u_{i+1}}_k\vert \le \Vert\di\Phi_{u_{i+1}}^{(i+1)}\Vert_{L^{\infty}_x}\vert \alpha^{u_i}_k\vert
	\le e^{M \vert u_{i+1}\vert}\vert \alpha^{u_i}_k\vert
	\quad\Longrightarrow\quad\vert \alpha^{u_i}_k\vert \leq e^{M \vert\uuu_i\vert_{\ell^1}}.
\end{equation}
In addition, if $i=k$, we use that $M \leq 1$ and \eqref{xi VS zi} to obtain
\begin{equation}\label{casek}
	\vert \alpha_{k}^{t_k}-\alpha_{k}^{s_k} \vert = \vert \cv^{(k)}(x_{t_k}) - \cv^{(k)}(x_{s_k}) \vert 
	\leq M |x_{t_k}-x_{s_k}| \leq 2^{d+1}\eta e^{M(\vert\ttt\vert_{\ell^1}+d\eta)}.
\end{equation}
We use induction to obtain a general bound. For $i=k,\ldots,d-1$, we write
\begin{multline*}
	\vert \alpha^{t_{i+1}}_k - \alpha^{s_{i+1}}_k\vert 
	= \vert\di\Phi^{(i+1)}_{t_{i+1}}(x_{t_i})\cdot \alpha_k^{t_i} - \di\Phi^{(i+1)}_{s_{i+1}}(x_{s_i})\cdot\alpha_k^{s_i}\vert
   	\le \vert\di\Phi^{(i+1)}_{t_{i+1}}(x_{t_i})\cdot \alpha_k^{t_i} - \di\Phi^{(i+1)}_{s_{i+1}}(x_{t_i})\cdot\alpha_k^{t_i}\vert \\[2pt]
   	+ \vert\di\Phi^{(i+1)}_{s_{i+1}}(x_{t_i})\cdot\alpha_k^{t_i} - \di\Phi^{(i+1)}_{s_{i+1}}(x_{t_i})\cdot\alpha_k^{s_i}\vert
	+ \vert\di\Phi^{(i+1)}_{s_{i+1}}(x_{t_i})\cdot\alpha_k^{s_i} - \di\Phi^{(i+1)}_{s_{i+1}}(x_{s_i})\cdot \alpha_k^{s_i}\vert .
\end{multline*}
We can bound the three above terms as follows. First of all, using the mean value theorem, the fact that 
\[
\Vert\di^2_{s,x}\Phi^{(i+1)}_{u_i}(x)\Vert = \Vert\di\cv^{(i+1)}(\Phi^{(i+1)}_{u_i}(x)) \circ \di \Phi^{(i+1)}_{u_i}(x)\Vert\le e^{\lambda |u_i|},
\]
the estimate \eqref{bound11}, and the fact that $\sss\in B(\ttt,\eta)$, we get 
\begin{multline*}
	\vert\di\Phi^{(i+1)}_{t_{i+1}}(x_{t_i})\cdot\alpha_k^{t_i} - \di\Phi^{(i+1)}_{s_{i+1}}(x_{t_i})\cdot\alpha_k^{t_i}\vert \\
	\le \sup_{u_i\in[t_{i+1},s_{i+1}]}\Vert\di^2_{s,x}\Phi^{(i+1)}_{u_i}\Vert_{L^{\infty}_x}\vert t_{i+1}-s_{i+1}\vert\vert\alpha^{t_i}_k\vert 
	\le \eta e^{ M ( \vert \ttt_{i+1} \vert_{\ell^1}+\eta)}.
\end{multline*}
Then, we deduce from \eqref{bound dPhi} that 
\[
	 \vert\di\Phi^{(i+1)}_{s_{i+1}}(x_{t_i})\cdot\alpha_k^{t_i} - \di\Phi^{(i+1)}_{s_{i+1}}(x_{t_i})\cdot\alpha_k^{s_i}\vert 
	 \le \Vert\di\Phi^{(i+1)}_{s_{i+1}}\Vert_{L^{\infty}_x}\vert\alpha^{t_i}_k-\alpha^{s_i}_k\vert 
	 \le e^{M\vert s_{i+1}\vert}\vert\alpha^{t_i}_k-\alpha^{s_i}_k\vert.
\]
Finally, it follows from the formula \eqref{diff2Flow}, \eqref{Me^M sigma}, \eqref{xi VS zi}, and \eqref{bound11} that
\begin{multline*}
    \vert\di\Phi^{(i+1)}_{s_{i+1}}(x_{t_i})\cdot\alpha_k^{s_i} - \di\Phi^{(i+1)}_{s_{i+1}}(x_{s_i})\cdot\alpha_k^{s_i}\vert 
    \\
	\le \Vert\di^2\Phi^{(i+1)}_{s_{i+1}}\Vert_{L^{\infty}_x}\vert x_{t_i}-x_{s_i}\vert\vert\alpha^{s_i}_k\vert
	\le 2^{d+1}\eta e^{M(\vert\sss_{i+1}\vert_{\ell^1}+ \vert\ttt_{i+1}\vert_{\ell^1}+(d+1)\eta)}.
\end{multline*}
Gathering the above estimates, we obtain
\[
	\vert\alpha^{t_{i+1}}_k - \alpha^{s_{i+1}}_k\vert\le e^{M\vert s_{i+1}\vert}\vert\alpha^{t_i}_k-\alpha^{s_i}_k\vert + 2^{d+2}\eta e^{M(\vert\sss_{i+1}\vert_{\ell^1}+ \vert\ttt_{i+1}\vert_{\ell^1}+(d+1)\eta)}.
\]
As a consequence, using induction for any $i=k,\ldots,d$ and \eqref{casek}, we have
\begin{align} \nonumber
	\vert\alpha^{t_i}_k - \alpha^{s_i}_k\vert
	&\le e^{M\vert\sss_i\vert_{\ell^1}} \vert\alpha^{t_k}_k - \alpha^{s_k}_k\vert + 2^{d+i+2}\eta e^{M(\vert\sss_i\vert_{\ell^1}+ \vert\ttt_i\vert_{\ell^1}+(d+1)\eta)} \\ \label{bound2}
	&\le 2^{2d+3}\eta e^{M(2\vert\ttt_i\vert_{\ell^1}+(2d+1)\eta)}.
\end{align}
We now have all the estimates we need on the vectors $\alpha^{u_i}_k$.

Let us now move to the estimates on the vectors $X_{u_i}$ and $\mathcal B_{u_i}$. First of all, notice that we have the following estimates from \eqref{bound11} and \eqref{bound2},
\[
	\vert\mathcal B_{u_i}\vert \leq \vert X_{u_i}\vert \leq (i+1)e^{M\vert\uuu_i\vert_{\ell^1}}
\]
and
\[
	\vert\mathcal B_{t_i} - \mathcal B_{s_i}\vert \leq \vert X_{t_i} - X_{s_i} \vert \leq (i+1)2^{2d+3}\eta e^{M(2\vert\ttt_i\vert_{\ell^1}+(2d+1)\eta)}.
\]
As a consequence,
\begin{align*}
	\vert\mathcal B_{t_i}\vert^2\vert\mathcal B_{t_i} - \mathcal B_{s_i}\vert(\vert\mathcal B_{t_i}\vert + \vert\mathcal B_{s_i}\vert) 
	& \le \vert X_{t_i}\vert^2\vert X_{t_i} - X_{s_i}\vert(\vert X_{t_i}\vert + \vert X_{s_i}\vert) \\
	& \le (i+1)^32^{2d+3}\eta e^{M(4\vert\ttt_i\vert_{\ell^1}+(2d+1)\eta)}
	((i+1)e^{M\vert\ttt_i\vert_{\ell^1}} + (i+1)e^{M\vert\sss_i\vert_{\ell^1}}) \\
	& \le 2^{6d+4}\eta e^{ M(5\vert\ttt_i\vert_{\ell^1}+4d\eta)}.
\end{align*}
As a consequence, under the assumption \eqref{eq:smallnessperturb}, the condition \eqref{eq:smalldet} holds and we deduce from the estimate \eqref{eq:avantder} that
\[
	\vert d_{t_i} - d_{s_i}\vert^2 \leq (d_{t_i}^2+d_{s_i}^2)2^{7d+5}\eta e^{ M(5 \vert\ttt_i\vert_{\ell^1}+4d\eta)}.
\]
Using the smallness condition \eqref{eq:smallnessperturb} again, we obtain 
\[
	2^{7d+5}\eta e^{M(5\vert\ttt_i\vert_{\ell^1}+4d\eta)}\leq\frac{1}{c^2},
\]
which gives the expected estimate \eqref{eq:interin}. This ends the proof of Proposition \ref{prop:perturb}.

\section{An inverse function theorem}\label{sec:local inversion}

In this third appendix, we state and prove a quantitative inverse function theorem on which the proof of Theorem \ref{th:AVLemmaII} is based. Adapting Theorem~2 of \cite{Ekeland2011} in our simple context, we have the following results.

\begin{proposition}\label{prop:from Ekeland}
Let $\eta>0$ and let $F\colon B(0,\eta)\to\R^d$ be a map of class $C^1$ satisfying the following hypothesis: $F(0)=0$, $\di F$ is invertible on $B(0,\eta)$, and there exists  $A_1> 0$ such that for all $\uuu\in B(0,\eta)$
\[
	\|\di F(\uuu)^{-1}\|\leq A_1.
\]
We claim that any $A_1'$ satisfying $A_1'>2A_1$ gives a control on the growth of a potential inverse function defined on a ball $B(0,\zeta_1)$. More precisely: if  $A_1'\zeta_1<\eta$, then, for all $z\in B(0,\zeta_1)$, the equation $F(\uuu)=z$ admits a solution $\uuu\in B(0,\eta)$ satisfying 
\begin{equation}\label{estimate on tau 1 first version}
	|\uuu|\leq A_1'|z|.
\end{equation}
\end{proposition}

The solutions to the equation $F(\uuu)=z$ are clearly not unique. For instance, to stick to our framework, consider, for $d=2$, $\cv^{(1)}(x,y)=(-x,y)$, $\cv^{(2)}(x,y)=(-y,x)$ and let $\Phi_t^{(i)}$ be the flow associated to $\cv^{(i)}$. On the picture below, we can consider the path $A\to B$ (via the integral curve of $\cv^{(1)}$), then $B\to A$ via the integral curve of $\cv^{(2)}$. This corresponds to a non-trivial solution $\uuu$ to the equation $\Phi_{u_2}^{(2)}\circ\Phi_{u_1}^{(1)}(A)=A$.

	\begin{center}
	\begin{tikzpicture}
		\begin{axis}[
			axis equal,
			xmin=-5, xmax=0,
			ymin=0, ymax=5,
			samples=25,
			xlabel=$x$,
			ylabel=$y$,
			grid=both,
			title={Integral curves of $\cv^{(1)}(x,y)=(-x,y)$ and $\cv^{(2)}(x,y)=(-y,x)$},
			domain=-5:5,
			]

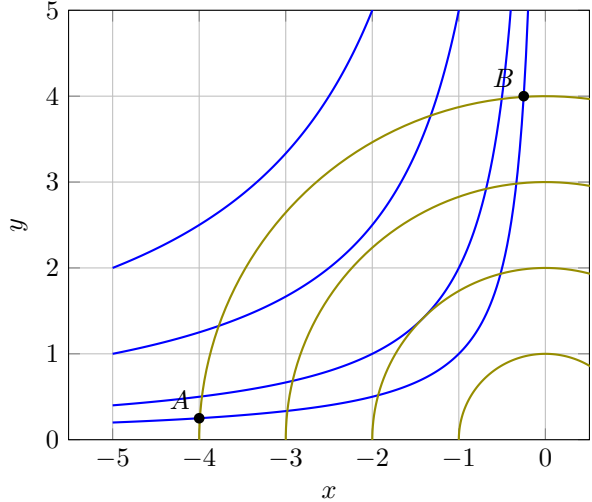
			\foreach \c in {-10, -5, -2, -1} {
				\addplot[
				domain=-5:-0.2,
				samples=200,
				thick,
				color=blue,
				] {\c/x};
				
				\addplot[
				domain=0.2:5,
				samples=200,
				thick,
				color=blue,
				] {\c/x};
			}

			\foreach \r in {1, 2, 3, 4} {
				\addplot[
				domain=0:360,
				samples=200,
				variable=\t,
				thick,
				color=olive,
				] ({\r*cos(\t)}, {\r*sin(\t)});
			}
			
			\fill (-4,1/4) circle (2pt) node[above left] {$A$};
			\fill (-1/4,4) circle (2pt) node[above left] {$B$};
			
			\end{axis}
	\end{tikzpicture}
\end{center}

\begin{proposition}\label{prop:from Ekeland Uniqueness} Under the hypotheses of Proposition~\ref{prop:from Ekeland}, assume that the Lipschitz property 
\begin{equation}\label{Lip DF}
	\|\di F(\uuu') - \di F(\uuu)\|\leq L_1|\uuu'-\uuu|,
\end{equation}
is satisfied for all $\uuu,\uuu'\in B(0,\eta)$, and that $\zeta_1$ satisfies the smallness condition
\begin{equation}\label{Lip DH for Uniq}
	A_1 A_1' \zeta_1 L_1<1.
\end{equation}
Then, if $z\in B(0,\zeta_1)$, the solution $\uuu\in B(0,A_1'\zeta_1)$ to the equation $F(\uuu)=z$ is unique.
\end{proposition}

\begin{proof}[Proof of Proposition~\ref{prop:from Ekeland}] To prove the existence of a solution satisfying \eqref{estimate on tau 1 first version}, we assume that $z \neq 0$ and we consider the minimum of the map 
\begin{equation}\label{minimization Ekeland}
	\sig\mapsto\frac12|F(\sig)-z|^2
\end{equation}
on $\bar{B}(0,r)$, where $r:=A_1'|z|$. By compactness, this minimum is reached at a point $\uuu$. We have 
\[
	F(\uuu) = \di F(\theta\uuu)\uuu,
\]
for a given $\theta\in[0,1]$, so by inversion and by the bound $\|\di F(\theta\uuu)^{-1}\|\leq A_1$, valid since $|\theta\uuu|\leq A_1'\zeta_1<\eta$, the modulus of $\uuu$ can be bounded as follows:
\[
	|\uuu|\leq A_1 |F(\uuu)|.
\]
Next, by optimality of $\uuu$,
\[
    |F(\uuu)|\leq|F(\uuu)-z| +|z|\leq|F(0)-z|+|z|=2|z|,
\]
and we get
\begin{equation}\label{toto estimate on tau 1}
	|\uuu|\leq 2A_1 |z|<r,
\end{equation}
which shows that $\uuu$ is interior to $\bar{B}(0,r)$. By differentiation in \eqref{minimization Ekeland}, we infer that
\[
	0 = \di F(\uuu)^*(F(\uuu)-z),
\]
where $\di F(\uuu)^*$ is the adjoint of $\di F(\uuu)$, and so invertible as well: we conclude that $F(\uuu)=z$. The bound  \eqref{estimate on tau 1 first version} follows from \eqref{toto estimate on tau 1}.
\end{proof}

\begin{proof}[Proof of Proposition~\ref{prop:from Ekeland Uniqueness}] We use \eqref{preLip DH} with $F$ instead of $H$ and $z=z_*$, to get
\[
	\vvv-\uuu = -\int_0^1 \di F(\uuu)^{-1}(\di F(\theta\vvv+(1-\theta)\uuu)-\di F(\uuu))(\vvv-\uuu)\,\mathrm d\theta,
\]
and thus
\[
	|\vvv-\uuu|\leq \frac12 A_1 L_1 |\vvv-\uuu|^2\leq A_1 A_1' \zeta_1 L_1 |\vvv-\uuu|,
\]
so $\vvv=\uuu$ if \eqref{Lip DH for Uniq} is satisfied.
\end{proof}

\section{Vector fields with constant coefficients}\label{sec:aconstant}

In this last appendix, we give the proof of the classical averaging lemma for vector fields $\cv$ not depending on the space variable $x\in\mathbb R^d$, in the case $p=2$. In this very particular situation, we can make use of Fourier analysis tools, which allows to obtain a sharp estimate with very few arguments.

\begin{proposition}\label{a inde}
Assume that $p=2$ and that the random vector field $\cv$ does not depend on the space variable $x\in\mathbb R^d$. Then, the modulus of continuity of the function $\varrho$ enjoys the following estimate for some positive constant $\kappa>0$,
\begin{equation}\label{opti 2}
	\forall r>0, \quad\sup_{|y|\le r}\|\tau_y\varrho-\varrho\|_{L^2(\mathbb R^d)}\leq \kappa \Vert g\Vert_{L^2(\Omega\times\mathbb R^d)} r^{\alpha/2}.	
\end{equation}
\end{proposition}

\begin{proof}
We will actually prove a stronger result, namely, an estimate for the function $\varrho$ in the Sobolev space $H^{\alpha/2}(\mathbb R^d)$,
\begin{equation}\label{eq:opti sobolev regularity}
    \|\varrho\|_{H^{\alpha/2}_x}\lesssim\|g\|_{L^2_{\omega,x}},
\end{equation}
which implies the inequality \eqref{opti 2} since the following bound holds for every $r>0$,
\begin{equation}\label{bound modulus continuity}
	\sup_{|y|\le r}\|\tau_y\varrho-\varrho\|_{L^2_x}\lesssim\Vert\varrho\Vert_{H^{\alpha/2}_x}r^{\alpha/2}.
\end{equation}
Notice that \eqref{bound modulus continuity} can be derived by using Plancherel's theorem and the inequality
\[
	\forall\theta\in[0,1], \forall x,y,\xi\in\mathbb R^d,\quad \vert e^{ix\cdot\xi}-e^{iy\cdot\xi}\vert\le 2^{1-\theta}\vert x-y\vert^{\theta}\vert\xi\vert^{\theta}.
\]
First of all, as a consequence of \eqref{averaged contraction Lp} (with $\lambda=0$), we have
\[
    \int_{\R^d} \left(\langle\xi\rangle^{\alpha}\wedge 2\right) |\hat{\varrho}(\xi)|^2\,\mathrm d\xi
    \lesssim\|\varrho\|_{L^2_x}^2\lesssim\|g\|^2_{L^2_{\omega,x}}.
\]
Consequently, it is sufficient to prove that
\begin{equation}\label{highmode}
    \int_{|\xi|\geq 1} |\xi|^{\alpha} |\hat{\varrho}(\xi)|^2\,\mathrm d\xi
    \lesssim\|g\|_{L^2_x}^2.
\end{equation}
After taking the Fourier transform, the equation \eqref{f plus agrad f} reads
\[
    (1-2\pi i a\cdot\xi)\hat{f}(\xi)=\hat{g}(\xi)
\]
and so, for all $\xi\in\R^d$,
\[
    \hat{\varrho}(\xi)=\E\left[\frac{\hat{g}(\xi)}{1-2\pi i a\cdot\xi}\right].
\]
By the Cauchy-Schwarz inequality, it follows that
\[
    |\hat{\varrho}(\xi)|^2\leq 2\E\left[\frac{1}{1+4\pi^2 |a\cdot\xi|^2}\right]\E\left[|\hat{g}(\xi)|^2\right].
\]
We will show that
\begin{equation}\label{GoodEstimate}
   \forall|\xi|\geq 1,\quad \E\left[\frac{1}{1+4\pi^2 |a\cdot\xi|^2}\right]\lesssim \frac{1}{|\xi|^\alpha}.
\end{equation}
which will yield \eqref{highmode}. To prove \eqref{GoodEstimate}, we consider the expansion
\begin{equation}\label{DecGoodEstimate}
    \E\left[\frac{1}{1+4\pi^2 |a\cdot\xi|^2}\right]=\sum_{j\in\Z}\E\left[\frac{\mathbbm{1}_{\{2^j <|a\cdot\xi|\leq 2^{j+1}\}}}{1+4\pi^2 |a\cdot\xi|^2}\right].
\end{equation}
The sum in \eqref{DecGoodEstimate} is on $j\leq J$ with $2^{J}=M|\xi|$, where $M$ is the a.s. bound for $\cv$ in \eqref{definition norme}, and can be estimated with the non-degeneracy condition \eqref{nd-a} by 
\begin{equation}\label{DecGoodEstimate-2}
    \sum_{j\leq J}\frac{\PP(|a\cdot\xi|\leq 2^{j+1})}{1+4\pi^2 2^{2j}}
    \leq C2^\alpha|\xi|^{-\alpha}\sum_{j\leq J}\frac{2^{\alpha j}}{1+4\pi^2 2^{2j}}.
\end{equation}
The sum over $j$ in \eqref{DecGoodEstimate-2} is bounded as follows:
\[
    \sum_{j\leq 0}\frac{2^{\alpha j}}{1+4\pi^2 2^{2j}}\leq \sum_{j\leq 0} 2^{\alpha j}\lesssim 1,
\]
while
\[
    \sum_{0<j\leq J}\frac{2^{\alpha j}}{1+4\pi^2 2^{2j}}\lesssim
    \sum_{0<j}2^{(\alpha-2)j}\lesssim 1.
\]
This ends the proof of the estimate \eqref{eq:opti sobolev regularity}.
\end{proof}

\medskip

{\small
\noindent\textsc{(Paul Alphonse) Univ Toulouse, IMT, Toulouse, France} \\
\textit{Email address}: \verb|paul.alphonse@math.univ-toulouse.fr|
}
\newline

{\small
\noindent\textsc{(Billel Guelmame) New York University Abu Dhabi, UAE} \\
\textit{Email address}: \verb|billel.guelmame@nyu.edu|
}
\newline

{\small
\noindent\textsc{(Julien Vovelle) CNRS, Universit\'e de Lyon, ENSL, UMPA - UMR 5669, F-69364 Lyon, France} \\
\textit{Email address}: \verb|julien.vovelle@ens-lyon.fr|
}

\end{document}